\documentclass[12pt, reqno]{amsart}
\usepackage{amssymb,amsthm,amsmath}
\usepackage[numbers,sort&compress]{natbib}
\usepackage{amssymb,amsmath}
\usepackage{amsfonts}
\usepackage{mathrsfs}
\usepackage{latexsym}
\usepackage{amssymb}
\usepackage{amsthm}
\usepackage{indentfirst}
\date{\today}

\let\oldsection\section
\renewcommand\section{\setcounter{equation}{0}\oldsection}

\newtheorem{corollary}{Corollary}[section]
\newtheorem{theorem}{Theorem}[section]

\newtheorem{proposition}{Proposition}[section]

\newtheorem{remark}{Remark}[section]

\allowdisplaybreaks
\begin{document}

\title[Global strong solutions 1D compressible Navier-Stokes]{Global well-posedness of the 1D compressible Navier-Stokes equations with constant heat conductivity and nonnegative density}

\author{Jinkai~Li}
\address[Jinkai~Li]{Department of Mathematics, The Chinese University of Hong Kong, Hong Kong, China}
\email{jklimath@gmail.com}


\keywords{Compressible Navier-Stokes equations; heat conductivity; vacuum.}
\subjclass[2010]{35A01,
35B45, 76N10, 76N17.}


\begin{abstract}
In this paper we consider the initial-boundary value problem to
the one-dimensional compressible Navier-Stokes
equations for idea gases. Both the viscous and heat conductive
coefficients are assumed to be positive constants, and the initial
density is allowed to have vacuum. Global existence and uniqueness of
strong solutions is established for any $H^2$ initial data, which
generalizes the well-known result of Kazhikhov--Shelukhin
(Kazhikhov, A.~V.; Shelukhin, V.~V.: \emph{Unique global solution with
respect to time of initial boundary value problems for one-dimensional
equations of a viscous gas}, J.\,Appl.\,Math.\,Mech., \bf41 \rm(1977),
273--282.) to the case that with nonnegative initial density.
An observation to overcome the difficulty caused by the lack of the positive lower bound of the density is that the ratio of the
density to its initial value is inversely proportional to the time integral of the upper bound of the temperature, along the trajectory. 
\end{abstract}

\maketitle

\allowdisplaybreaks

\section{Introduction}
\subsection{The compressible Navier-Stokes equations}
In this paper, we consider the following one-dimensional heat conductive
compressible Navier-Stokes equations:
\begin{eqnarray}
  \partial_t\rho+\partial_x(\rho u)&=&0,\label{1.1}\\
  \rho(\partial_tu+u\partial_xu)-\mu\partial_x^2u+\partial_xp&=&0,
  \label{1.2}\\
  c_v\rho(\partial_t\theta+u\partial_x\theta)+\partial_xup
  -\kappa\partial_x^2\theta&=&\mu(\partial_xu)^2, \label{1.3}
\end{eqnarray}
where $\rho, u, \theta$, and $p$, respectively, denote the density,
velocity, absolute temperature, and pressure. The viscous coefficient
$\mu$ and heat conductive coefficient $\kappa$ are assumed to be
positive constants. The state equation for the ideal gas reads as
$$
p=R\rho\theta,
$$
where $R$ is a positive constant.

The compressible Navier-Stokes equations have been extensively studied.
In the absence of vacuum, i.e., the case that the initial density
is uniformly bounded away from zero, global well-posedness
of strong solutions to the one dimensional compressible
Navier-Stokes equations has been
well-known since the pioneer works by Kazhikhov--Shelukhin
\cite{KAZHIKOV77} and Kazhikhov \cite{KAZHIKOV82}.
Inspired by these works, global existence and uniqueness of weak
solutions were later established by Zlotnik--Amosov
\cite{ZLOAMO97,ZLOAMO98} and  Chen--Hoff--Trivisa \cite{CHEHOFTRI00} for
the initial boundary
value problems, and by Jiang--Zlotnik \cite{JIAZLO04} for the Cauchy
problem. Large time behavior of solutions to the one dimensional
compressible Navier-Stokes equations with large initial data was
recently proved by Li--Liang \cite{LILIANG16}. The corresponding global
well-posedness results for the
multi-dimensional case were established only for
small perturbed initial data around some non-vacuum equilibrium or for
spherically symmetric large initial data, see, e.g., Matsumura--Nishida
\cite{MATNIS80,MATNIS81,MATNIS82,MATNIS83}, Ponce \cite{PONCE85},
Valli--Zajaczkowski \cite{VALZAJ86}, Deckelnick \cite{DECK92}, Jiang
\cite{JIANG96}, Hoff \cite{HOFF97}, Kobayashi--Shibata \cite{KOBSHI99},
Danchin \cite{DANCHI01}, Chikami--Danchin \cite{CHIDAN15}, and the
references therein.

In the presence of vacuum, that is the density may vanish
on some set or tends to zero at the far field, global existence of
weak solutions to the isentropic compressible Navier-Stokes equations
was first proved by Lions \cite{LIONS93,LIONS98},
with adiabatic constant $\gamma\geq\frac95$, and later generalized by
Feireisl--Novotn\'y--Petzeltov\'a \cite{FEIREISL01} to $\gamma>\frac32$,
and further by Jiang--Zhang \cite{JIAZHA03} to $\gamma>1$
for the axisymmetric solutions. For the full compressible Navier-Stokes
equations, global existence of the variational weak solutions was proved
by Feireisl \cite{FEIREISL04P,FEIREISL04B}, which however is not
applicable for the ideal gases. Local well-posedness of strong solutions
to the full compressible
Navier-Stokes equations, in the presence of vacuum, was proved by
Cho--Kim \cite{CHOKIM06-2}, see also Salvi--Stra$\check{\text s}$kraba
\cite{SALSTR93}, Cho--Choe--Kim \cite{CHOKIM04}, and Cho--Kim
\cite{CHOKIM06-1} for the isentropic case. The solutions in
\cite{SALSTR93,CHOKIM04,CHOKIM06-1,CHOKIM06-2} are established
in the homogeneous Sobolev spaces, and, generally, one can not
expect the solutions in
the inhomogeneous Sobolev spaces, if the initial density has compact
support, due to the recent nonexistence result by Li--Wang--Xin
\cite{LWX}. Global existence of strong solutions to the compressible
Navier-Stokes equations, with small initial
data, in the presence of initial vacuum, was first proved by Huang--Li--Xin \cite{HLX12} for the isentropic case (see also Li--Xin
\cite{LIXIN13} for further developments), and later by
Huang--Li \cite{HUANGLI11} and Wen--Zhu \cite{WENZHU17} for the full
case. Due to the finite blow-up results in \cite{XIN98,XINYAN13},
the global solutions
obtained in \cite{HUANGLI11,WENZHU17} must have
infinite entropy somewhere in the vacuum region, if the initial density
has an isolated mass group; however, if the initial density is positive
everywhere but tends to vacuum at the far field, one can expect the
global existence of solutions with uniformly bounded entropy to the full
compressible Navier-Stokes equations, see the recent work by the author and Xin \cite{LIXIN17}.

Note that in the global well-posedness results for system
(\ref{1.1})--(\ref{1.3}) in
\cite{KAZHIKOV77,KAZHIKOV82,ZLOAMO97,ZLOAMO98,CHEHOFTRI00,JIAZLO04,LILIANG16},
the density was assumed to be uniformly away from vacuum.
Global well-posedness of strong solutions to system
(\ref{1.1})--(\ref{1.3}) in the presence of vacuum was proved by
Wen--Zhu \cite{WENZHU13}; however, due to the following assumption
$$
\kappa_0(1+\theta^q)\leq\kappa(\theta)\leq\kappa_1(1+\theta^q),\quad\mbox{ for some }q>0,
$$
made on the heat conductive coefficient $\kappa$ in \cite{WENZHU13},
the case that $\kappa(\theta)\equiv const.$\,was not included there.
The aim of this paper is to study the global well-posedness of strong
solutions to system (\ref{1.1})--(\ref{1.3}), with both constant
viscosity and constant diffusivity, in the presence of vacuum.
This result will be proven
in the Lagrangian flow map coordinate; however, it can be
equivalently translated back to the corresponding one in the
Euler coordinate.

\subsection{The Lagrangian coordinates and main result} Let $y$ be the Lagrangian coordinate, and define the coordinate
transform
between the Lagrangian coordinate $y$ and the Euler coordinate $x$ as
\begin{equation*}
  x=\eta(y,t),
\end{equation*}
where
$\eta(y,t)$ is the flow map determined by $u$, that is
\begin{equation*}\label{flowmap}
  \left\{
  \begin{array}{l}
  \partial_t\eta(y,t)=u(\eta(y,t),t),\\
  \eta(y,0)=y.
  \end{array}
  \right.
\end{equation*}

Denote by $\varrho, v, \vartheta$, and $\pi$ the density, velocity,
temperature, and pressure, respectively, in the Lagrangian coordinate,
that is we define
\begin{eqnarray*}
  \varrho(y,t):=\rho(\eta(y,t),t),\quad v(y,t):=u(\eta(y,t),t), \\
  \vartheta(y,t):=\theta(\eta(y,t),t),\quad \pi(y,t):=p(\eta(y,t),t).
\end{eqnarray*}
Recalling the definition of $\eta(y,t)$, by straightforward
calculations, one can check that
\begin{eqnarray*}
  &(\partial_xu,\partial_x\theta,\partial_xp)
  =\left(\frac{\partial_yv}{\partial_y\eta},
  \frac{\partial_y\vartheta}{\partial_y\eta},\frac{
  \partial_y\pi}{\partial_y\eta}\right),\quad
  (\partial_x^2u,\partial_x^2\vartheta)
  =\left(\frac{1}{\partial_y\eta}\partial_y\left(
  \frac{\partial_yv}{\partial_y\eta}\right),
  \frac{1}{\partial_y\eta}\partial_y\left(
  \frac{\partial_y\vartheta}{\partial_y\eta}\right)\right)\\
  &\partial_t\rho+u\partial_x\rho=\partial_t\varrho,\quad
  \partial_tu+u\partial_xu=\partial_tv,\quad
  \partial_t\theta+u\partial_x\theta=\partial_t\vartheta.
\end{eqnarray*}

Define a function $J=J(y,t)$ as
\begin{equation*}
  J(y,t)=\eta_y(y,t),
\end{equation*}
then it follows
\begin{equation}
\label{eqJ}
\partial_t J=\partial_yv,
\end{equation}
and system (\ref{1.1})--(\ref{1.3}) can be rewritten
in the Lagrangian coordinate as
\begin{eqnarray}
  \partial_t\varrho+\frac{\partial_y v}{J}\varrho&=&0,\label{eqrho}\\
  \varrho\partial_tv-\frac{\mu}{J}\partial_y\left(\frac{\partial_yv}{J}
  \right)+\frac{\partial_y \pi}{J}&=&0,\label{eqv}\\
  c_v\varrho\partial_t\vartheta+\frac{\partial_yv}{J}\pi
  -\frac{\kappa}{J}\partial_y\left(\frac{\partial_y\vartheta}{J}\right)
  &=&\mu\left(\frac{\partial_y v}{J}\right)^2,\label{eqtheta}
\end{eqnarray}
where $\pi=R\varrho\vartheta.$

Due to (\ref{eqJ}) and (\ref{eqrho}), it holds that
$$
\partial_t(J\varrho)=\partial_t J\varrho+J\partial_t\varrho=
\partial_y v\varrho-J\frac{\partial_yv}{J}\varrho=0,
$$
from which, by setting $\varrho|_{t=0}=\varrho_0$ and noticing that
$J|_{t=0}=1$, we have
\begin{equation*}
  J\varrho=\varrho_0.
\end{equation*}
Therefore, one can replace (\ref{eqrho}) with (\ref{eqJ}), by setting
$\varrho=\frac{\varrho_0}{J}$, and rewrite (\ref{eqv}) and
(\ref{eqtheta}), respectively, as
\begin{equation*}
  \varrho_0\partial_tv-\mu\partial_y\left(\frac{\partial_yv}{J}\right)
  +\partial_y\pi=0
\end{equation*}
and
$$
c_v\varrho_0\partial_t\vartheta+\partial_yv\pi
-\kappa\partial_y\left(\frac{\partial_y\vartheta}{J}\right)
=\mu\frac{(\partial_y v)^2}{J}.
$$

In summary, we only need to consider the following system
\begin{eqnarray}
  \partial_tJ&=&\partial_yv,\label{EQJ}\\
  \varrho_0\partial_tv-\mu\partial_y\left(\frac{\partial_yv}{J}\right)
  +\partial_y \pi&=&0,\label{EQv}\\
  c_v\varrho_0\partial_t\vartheta+\partial_yv\pi
  -\kappa\partial_y\left(\frac{\partial_y\vartheta}{J}\right)
  &=&\mu\frac{(\partial_y v)^2}{J},\label{EQtheta}
\end{eqnarray}
where
$$
\pi=R\frac{\varrho_0}{J}\vartheta.
$$

We consider the initial-boundary value problem to system
(\ref{EQJ})--(\ref{EQtheta}) on the interval $(0,L)$, with $L>0$, that 
is system (\ref{EQJ})--(\ref{EQtheta}) is defined in the space-time
domain $(0,L)\times(0,\infty)$. We complement the system with the
following boundary and initial conditions:
\begin{equation}
  \label{BC}
  v(0,t)=v(L,t)=\partial_y\vartheta(0,t)=\partial_y\vartheta(L,t)=0
\end{equation}
and
\begin{equation}
  \label{IC}
  (J,v,\vartheta)|_{t=0}=(1,v_0,\vartheta_0).
\end{equation}

For $1\leq q\leq\infty$ and positive integer $m$, we use $L^q=L^q((0,L))$ and $W^{1,q}=W^{m,q}((0,L))$ to denote the standard Lebesgue and Sobolev spaces, respectively, and in the case that $q=2$, we use $H^m$
instead of $W^{m,2}$. 
We always use $\|u\|_q$ to denote the $L^q$ norm of $u$. 

The main result of this paper is the following:

\begin{theorem}
\label{thm}
  Given $(\varrho_0, v_0, \vartheta_0)\in H^2((0,L))$, satisfying \begin{eqnarray*}
    &&\varrho_0(y)\geq0,\quad\vartheta(y)\geq0,\quad \forall y\in[0,L],\\
    &&v_0(0)=v_0(L)=\partial_y\vartheta_0(0)=\partial_y\vartheta_0(L)=0.
  \end{eqnarray*}
  Assume that the following compatibility conditions hold
  \begin{eqnarray*}
    \mu v_0''+R(\varrho_0\vartheta_0)''=\sqrt{\varrho_0}g_0, \\
    \kappa\vartheta_0''+\mu(v_0')^2-Rv_0'\varrho_0\vartheta_0
    =\sqrt{\varrho_0}h_0,
  \end{eqnarray*}
  for two functions $g_0, h_0\in L^2((0,L))$.

  Then, there is a unique global solution $(J, v, \vartheta)$ to system (\ref{EQJ})--(\ref{EQtheta}), subject to (\ref{BC})--(\ref{IC}), satisfying $J>0$ and $\vartheta\geq0$ on $[0,L]\times[0,\infty)$, and
  \begin{eqnarray*}
    &&J\in C([0,T]; H^2),\quad\partial_tJ\in L^2(0,T; H^2),\\
    &&v\in C([0,T]; H^2)\cap L^2(0,T; H^3),
    \quad\partial_tv\in L^2(0,T; H^1),\\
    &&\vartheta\in C([0,T]; H^2)\cap L^2(0,T; H^3),\quad \partial_t\vartheta\in L^2(0,T; H^1),
  \end{eqnarray*}
  for any $T\in(0,\infty)$.
\end{theorem}

\begin{remark}
The same result as in Theorem \ref{thm} still holds if replacing the
boundary condition $\partial_y\vartheta(0,t)=\partial_y\vartheta(L,t)=0$
by one of the following three
\begin{eqnarray*}
  \vartheta(0,t)=\vartheta(L,t)=0,\\
  \vartheta(0,t)=\partial_y\vartheta(L,t)=0,\\
  \partial_y\vartheta(0,t)=\vartheta(L,t)=0, 
\end{eqnarray*}
and the proof is exactly the same as the one presented in this paper, 
the only different is that the basic energy identity in Proposition \ref{PROPBASIC} will then be an inequality. 
\end{remark}

\begin{remark}
  The argument presented in this paper also 
  works for the free boundary value
  problem to the same system. Because, if rewritten the system in the Lagrangian
  coordinates,
  the only difference between the initial boundary value problem and the
  free boundary value problem is the boundary conditions for $v$: in the free boundary problem, the boundary conditions for $v$ in (\ref{BC}) are replaced by
  $$
  \mu\frac{\partial_yv}{J}-\pi\Big|_{y=0,L}=0. 
  $$
  Note that all the energy estimates obtained in this paper hold if replacing the boundary condition on $v$ in (\ref{BC}) with the above ones, by slightly modifying the proof. 
\end{remark}

The argument used in Kazhikhov-Shelukhin \cite{KAZHIKOV77}, in which
the non-vacuum case was considered, does not apply directly to 
the vacuum case. One main observation 
in \cite{KAZHIKOV77} is: the lower bound of the density is inversely
proportional to the time integral of the upper bound of the temperature,
along the trajectory. Note that this only holds for the
case that the density has a positive uniform lower bound. 
To overcome the difficulty caused by the lack of the positive lower bound of the density (it is of this case if the lower bound
of the initial density is zero),
our observation is: the ratio of the
density to its initial value is inversely 
proportional to the time integral of the temperature, along the
trajectory, or, equivalently, the upper bound of
$J$ is proportional to the time integral of the
upper bound of the temperature. This 
observation holds for both the
vacuum and non-vacuum cases, which, in particular, reduces to the one
in \cite{KAZHIKOV77} for the non-vacuum case; this also indicates 
the advantage of taking $J$ rather than $\varrho$ as one the unknowns. 
 
The key issue of proving Theorem \ref{thm} is to establish the
appropriate
a priori energy estimates, up to any finite time, of the solutions to
system (\ref{EQJ})--(\ref{EQtheta}), subject to (\ref{BC})--(\ref{IC}).
There are four main stages for carrying out the desired a priori energy
estimates. In the first
stage, we derive from (\ref{EQJ})--(\ref{EQv}) an identity 
$$
1+\frac R\mu\varrho_0(y)\int_0^t\vartheta(y,\tau)
H(\tau)B(y,\tau)d\tau=J(y,t)H(t)B(y,t),
$$
for some functions $H(t)$ and $B(y,t)$. The temperature
equation is not used at all in deriving the above identity, and this
identity is in the spirit of the one in \cite{KAZHIKOV77}, but in
different Lagrangian coordinates. The basic energy estimate implies that
both $H$ and $B$ are uniformly away from
zero and uniformly bounded, up to any finite time. 
As a direct corollary of the above
identity, one can obtain the uniform positive
lower bound of $J$, and the control of the upper bound of $J$ in terms
of the time integral of $\vartheta$. By using the positive lower bound
of $J$, one obtains a density-weighted embedding inequality (which can
be viewed as the replacement of the Sobolev embedding inequality when
the vacuum is involved) for $\vartheta$, see (ii) of Proposition
\ref{PROPEstRhoWei}, which implies that the upper bound of
$\sqrt{\varrho_0}\vartheta$ can be
controlled by that of $J$, up to a small dependence on
$\|\frac{\partial_y\vartheta}{\sqrt J}\|_2$, i.e., the term
$\eta\|\frac{\partial_y\vartheta}{\sqrt J}\|_2$, with a small positive
$\eta$. This combined with the above identity leads
us to carry out the $L^\infty(L^2)$ type estimates 
on $\vartheta$, or, more precisely, on $\sqrt{\varrho_0}\vartheta$. 
In the second stage, we carry out the $L^\infty(L^2)$ energy estimate 
on $\sqrt{\varrho_0}\vartheta$, and, at the same time, the
$L^\infty(L^2)$ energy estimate will be involved naturally, due to the 
coupling structure 
between $v$ and $\vartheta$ in the system; 
as a conclusion of this stage,
by making use the control relationship between the upper bounds of
$\sqrt{\varrho_0}\vartheta$ and $J$ obtained in the first stage, we are
able to obtain the a priori upper bound of $J$ and the a priori
$L^\infty(L^2)\cap L^2(H^1)$ type estimates on $(v,\vartheta)$. In the
third stage, by investigating the effective viscous flux
$G:=\mu\frac{\partial_yv}{J}-\pi$ and working
on its $L^\infty(L^2)\cap L^2(H^1)$ type a priori estimate, we are 
able to get the a priori $L^\infty(H^1)$ estimate on $(J, v)$; however,
due to the presence of the term $\frac{\mu}{J}(\partial_yv)^2$
and the degeneracy of the leading term $\varrho_0\partial_t\vartheta$ in
the $\vartheta$ equation, we are not able to obtain the
corresponding $L^\infty(H^1)$ estimate on $\vartheta$, without appealing
to higher order energy estimates than $H^1$. In the fourth stage, we
carry out the a priori $L^\infty(H^2)$ type estimates, which are
achieved through performing the $L^\infty(L^2)$ type energy estimate
on $\sqrt{\varrho_0}\partial_t\vartheta$ and $L^\infty(H^1)$ type
estimate on $G$. It should be mentioned that the desired a priori $L^\infty(H^2)$ estimates on $\vartheta$ is obtained without knowing its a priori $L^\infty(H^1)$ bound in advance.

The rest of this paper is arranged as follows: in the next section,
Section \ref{SecApri}, which is the main part of this paper,
we consider the global existence and the a priori estimates to system
(\ref{EQJ})--(\ref{EQtheta}), subject to (\ref{BC})--(\ref{IC}), in the
absence of vacuum, while Theorem \ref{thm} is proven in the last
section.

Throughout this paper, we use $C$ to denote a general positive constant which may different from line to line. 

\section{Global existence and a priori estimates in the absence of vacuum}
\label{SecApri}

We first recall the global existence results due to
Kazhikohov-Shelukhin \cite{KAZHIKOV77} stated in the following
proposition. The original
result in \cite{KAZHIKOV77} was stated in the Lagrangian mass
coordinates,
rather than the Lagrangian map flow coordinates as here, and
the initial data $(\varrho_0, v_0, \vartheta_0)$ was assumed in $H^1$;
however, due to the sufficient regularities of the solutions
established in \cite{KAZHIKOV77}, in particular the $L^1(0,T;
W^{1,\infty})$ of $v$, the global existence result established there can
be translated to the corresponding one in the Lagrangian map flow
coordinates, and one can show that the solutions have correspondingly
more regularities if the initial data has more regularities as stated
in the following proposition.

\begin{proposition}
  \label{PROPGLOBAL}
For any $(\varrho_0, v_0, \vartheta_0)\in H^2$, satisfying $v(0)=v(L)=0$, and
\begin{eqnarray*}
  \min\left\{\inf_{y\in(0,L)}\varrho_0(y),\inf_{y\in(0,L)}\vartheta_0(y)
  \right\}>0,
\end{eqnarray*}
there is a unique global solution $(J,v,\varrho)$ to system (\ref{EQJ})--(\ref{EQtheta}), subject to (\ref{BC})--(\ref{IC}), satisfying $J,\vartheta>0$ on $(0,L)\times(0,\infty)$, and
 \begin{eqnarray*}
J\in C([0,T]; H^2),&& \partial_tJ\in L^\infty(0,T; H^1)\cap L^2(0,T; H^2),\\
v\in C([0,T]; H^2)\cap L^2(0,T; H^3), && \partial_t v\in L^\infty(0,T; L^2)\cap L^2(0,T; H^1),\\
\vartheta\in C([0,T]; H^2)\cap L^2(0,T; H^3), && \partial_t\vartheta\in L^\infty(0,T; L^2)\cap L^2(0,T; H^1),
  \end{eqnarray*}
  for any $T\in(0,\infty)$.
\end{proposition}

In the rest of this section, we always
assume that $(J,v,\vartheta)$ is the
unique global solution obtained in Proposition \ref{PROPGLOBAL},
and we will establish a series of a priori
estimates of $(J,v,\vartheta)$ independent of the lower bound
of the density.

We start with the basic energy identity in the following proposition.

\begin{proposition}
\label{PROPBASIC}
It holds that
\begin{equation*}
  \int_0^LJ(y,t)dy=L
\end{equation*}
and
$$
\left(\int_0^L\left(\frac{\varrho_0}{2}v^2+c_v\varrho_0\vartheta\right)dy
\right)(t)=E_0,
$$
for any $t\in(0,\infty)$, where $E_0:=
\int_0^L\left(\frac{\varrho_0}{2}v_0^2+c_v\varrho_0\vartheta_0\right)dy$.
\end{proposition}

\begin{proof}
  The first conclusion follows directly from integrating (\ref{EQJ})
  with respect to $y$ over $(0,L)$ and using the boundary condition (\ref{BC}). Multiplying
  equation (\ref{EQv}) by $v$, integrating the resultant over $(0,L)$,
  one gets from integrating by parts that
  $$
  \frac12\frac{d}{dt}\int_0^L\varrho_0v^2dy+\mu\int_0^L\frac{(\partial_y v)^2}{J}dy =\int_0^L\partial_yv\pi dy.
  $$
  Integrating (\ref{EQtheta}) over $(0,L)$ yields
  $$
  c_v\frac{d}{dt}\int_0^L\varrho_0\vartheta dy+\int_0^L\partial_yv\pi dy
  =\mu\int_0^L\frac{(\partial_yv)^2}{J}dy,
  $$
  which, summed with the previous equality, leads to
  $$
  \frac{d}{dt}\int_0^L\left(\frac{\varrho_0}{2}v^2+c_v\varrho_0\vartheta
  \right)dy
  =0,
  $$
  the second conclusion follows.
\end{proof}

\subsection{A priori $L^2$ estimates} In this subsection, we
will derive the uniform positive
lower and upper bounds of $J$ and the
a priori $L^\infty(0,T; L^2)\cap
L^2(0,T; H^1)$ estimates of $v$ and $\vartheta$.

Before carrying out the desired estimates, we first
derive an equality, i.e.,
(\ref{EQJHB}) in the below, in
the spirit of \cite{KAZHIKOV77}. As will be shown later, this equality
leads to the positive lower bound of $J$, and it will be combined with
the $L^2$ type energy inequalities to get the upper bound of $J$ and the
a priori $L^\infty(0,T; L^2)\cap
L^2(0,T; H^1)$ estimates of $v$ and $\vartheta$.
Due to (\ref{EQJ}), it follows from (\ref{EQv}) that
$$
\varrho_0\partial_t v-\mu\partial_y\partial_t\log J+\partial_y\pi=0.
$$
Integrating the above equation with respect to $t$ over $(0,t)$ yields
$$
\varrho_0(v-v_0)-\mu\partial_y\log J+ \partial_y\left(\int_0^t\pi d\tau\right) =0,
$$
from which, integrating with respect to $y$ over $(z,y)$, one obtains
\begin{align*}
  \int_z^y\varrho_0(v-v_0)d\xi -\mu\left(\log J(y,t)-\log J(z,t) \right) +\int_0^t(\pi(y,\tau)-\pi(z,\tau))d\tau=0,
\end{align*}
for any $y,z\in(0,L)$.
Using
$$
\int_z^y\varrho_0(v-v_0)d\xi= \int_0^y\varrho_0(v-v_0)d\xi- \int_0^z\varrho_0(v-v_0)d\xi,
$$
and rearranging the terms, we obtain
\begin{eqnarray*}
  &&\int_0^y
  \varrho_0(v-v_0)d\xi-\mu\log J(y,t)+\int_0^t\pi(y,\tau)d\tau \\
  &=&\int_0^z
  \varrho_0(v-v_0)d\xi-\mu \log J(z,t)+\int_0^t\pi(z,\tau)d\tau,\qquad
  \forall y,z\in(0,L).
\end{eqnarray*}

Therefore, the function
$$
\int_0^y
  \varrho_0(v-v_0)d\xi-\mu\log J(y,t)+\int_0^t\pi(y,\tau)d\tau
$$
is independent of $y$, and we denote it by $h(t)$, that is
\begin{equation*}
  \int_0^y
  \varrho_0(v-v_0)d\xi-\mu\log J+\int_0^t\pi d\tau=h(t),
\end{equation*}
from which, recalling $\pi=R\frac{\varrho_0}{J}\vartheta$, one gets
\begin{equation}
  \frac1J\exp\left\{\frac R\mu\varrho_0\int_0^t\frac{\vartheta}{J}d\tau
  \right\}=H(t)B(y,t), \label{EQH}
\end{equation}
or equivalently
\begin{equation}
  \label{EQH'}
  \exp\left\{\frac R\mu\varrho_0\int_0^t\frac{\vartheta}{J}d\tau
  \right\}=J(y,t)H(t)B(y,t)
\end{equation}
where
$$
H(t)=\exp\left\{\frac{h(t)}{\mu}\right\},\quad
B(y,t)=\exp\left\{-\frac1\mu\int_0^y\varrho_0(v-v_0)d\xi\right\}.
$$

Multiplying (\ref{EQH}) by $\frac R\mu\varrho_0\vartheta$ one obtains
$$
\partial_t\left(\exp\left\{\frac R\mu\varrho_0\int_0^t\frac{\vartheta}{J}d\tau\right\} \right)=\frac R\mu\varrho_0\vartheta HB,
$$
which gives
$$
\exp\left\{\frac R\mu\varrho_0\int_0^t\frac{\vartheta}{J}d\tau\right\}
=1+\frac R\mu\varrho_0\int_0^t\vartheta HBd\tau.
$$
Combining this with (\ref{EQH'}), one gets
\begin{equation}
\label{EQJHB}
  1+\frac R\mu\varrho_0(y)\int_0^t\vartheta(y,\tau) H(\tau)B(y,\tau)d\tau=J(y,t)H(t)B(y,t),
\end{equation}
for any $y\in(0,L)$ and any $t\in[0,\infty)$.

A prior positive lower bound of $J$ and the control of the
upper bound of $J$ in terms of $\varrho_0\vartheta$ are stated in the
following proposition:

\begin{proposition}
We have the following estimate:
  \label{PROPEstJ}
\begin{eqnarray*}
  (m_1f_1(t))^{-1}\leq J(y,t)\leq m_1^2+\frac R\mu
  m_1^3f_1(t)\int_0^t\varrho_0\vartheta d\tau,
\end{eqnarray*}
for any $y\in(0,L)$ and any $t\in[0,\infty)$, where
$$
m_1=\exp\left\{\frac2\mu\sqrt{2\|\varrho_0\|_1E_0}\right\},\quad
f_1(t)=m_1\exp\left\{
  \frac{Rm_1^2E_0t}{\mu c_vL}\right\}.
$$
\end{proposition}

\begin{proof}
By Proposition \ref{PROPBASIC}, it follows from the H\"older inequality
that
\begin{eqnarray*}
  \left|\int_0^y\varrho_0(v-v_0)d\xi\right|\leq\int_0^L(|\varrho_0v|+|
  \varrho_0v_0|)d\xi
  \leq 2\|\varrho_0\|_1^{\frac12}(2E_0)^{\frac12}
\end{eqnarray*}
and, thus,
\begin{equation}
\label{EstB}
\underline m:=\exp\left\{-\frac2\mu\sqrt{2\|\varrho_0\|_1E_0}
\right\}
\leq B(y,t)\leq \exp\left\{\frac2\mu\sqrt{2\|\varrho_0\|_1E_0}
\right\}=: m_1.
\end{equation}

Integrating (\ref{EQJHB}) over $(0,L)$, it follows from (\ref{EstB})
and Proposition \ref{PROPBASIC} that
\begin{eqnarray*}
  L&\leq&\int_0^L\left(1+\frac R\mu\varrho_0\int_0^t\vartheta HBd\tau\right)dy=H(t)\int_0^LJBdy \\
  &\leq&m_1 H(t) \int_0^LJdy=m_1LH(t)
\end{eqnarray*}
and
\begin{eqnarray*}
  \underline mLH(t)&=&H(t)\underline m\int_0^LJdy\\
  &\leq&H(t)\int_0^LJBdy
   = \int_0^L\left(1+\frac R\mu\varrho_0\int_0^t\vartheta HBd\tau\right)dy\\
  &\leq&L+\frac{R m_1}{\mu}\int_0^t\int_0^L\varrho_0\vartheta dyH d\tau
  \leq L+\frac{Rm_1E_0}{\mu c_v}\int_0^tHd\tau.
\end{eqnarray*}
Therefore, we have
\begin{eqnarray*}
   m_1^{-1}\leq H(t)\leq \underline m^{-1}+\frac{Rm_1E_0}{\mu c_v\underline mL}\int_0^tHd\tau,
\end{eqnarray*}
from which, by the Gronwall inequality, one further obtains
\begin{equation}
\label{EstH}
   m_1^{-1}\leq H(t)\leq \underline m^{-1}\exp\left\{
  \frac{Rm_1E_0t}{\mu c_v\underline mL}\right\}=: f_1(t),\quad t\in[0,T].
\end{equation}

Due to (\ref{EstB}) and (\ref{EstH}), it follows from (\ref{EQJHB}) that
\begin{eqnarray*}
  1\leq 1+\frac R\mu\varrho_0\int_0^t\vartheta HBd\tau=J(y,t)H(t)B(y,t)
  \leq J(y,t) m_1 f_1(t),
\end{eqnarray*}
and
\begin{eqnarray*}
   m_1^{-1}\underline mJ(y,t)&\leq&H(t)B(y,t)J(y,t)
  =1+\frac R\mu\varrho_0\int_0^t\vartheta HBd\tau\\
  &\leq&1+\frac R\mu m_1 f_1(t)\int_0^t\varrho_0\vartheta d\tau.
\end{eqnarray*}
Therefore, we have
\begin{eqnarray*}
  ( m_1 f_1(t))^{-1}\leq J(y,t)\leq  m_1\underline m^{-1}+\frac{R m_1^2 f_1(t)}{\mu\underline m}\int_0^t\varrho_0\vartheta d\tau,
\end{eqnarray*}
for any $y\in(0,L)$ and any $t\in[0,T]$, proving the conclusion.
\end{proof}

As a preparation of deriving the a priori upper bound of $J$ and the
a priori $L^\infty(0,T; L^2)\cap L^2(0,T; H^1)$ type estimates on
$(v,\vartheta)$, we prove the following proposition, which,
in particular, gives the density-weighted estimate of $\vartheta$.

\begin{proposition}
  \label{PROPEstRhoWei}
We the following two items:

(i) It holds that
\begin{eqnarray*}
\|\varrho_0^2\vartheta\|_\infty^2&\leq&\left(\frac{E_0}{c_v}\right)^2
\left(\frac{8\bar\varrho^2}{L^2}+32\|\varrho_0'\|_\infty^2\right)
+6\bar\varrho^{\frac{10}{3}}\left(\frac{E_0}{c_v}\right)^{\frac23}
\left\|\frac{\partial_y\vartheta}
{\sqrt J}\right\|_2^{\frac43}\|J\|_\infty^{\frac23},\\
\|\vartheta\|_\infty&\leq&\sqrt L
\left\|\frac{\partial_y\vartheta}
{\sqrt J}\right\|_2+\frac{2E_0}{c_v\omega_0\bar\varrho},
\end{eqnarray*}
where
$$
\bar\varrho=\|\varrho_0\|_\infty,\quad\omega_0=|\Omega_0|,
\quad\Omega_0:=\left\{y\in(0,L)\Big|\varrho_0(y)\geq \frac{\bar\varrho}{2}\right\}.
$$

(ii) As a consequence of (i), we have
$$
\|\sqrt{\varrho_0}\vartheta\|_\infty^2 \leq\eta\left\|\frac{\partial_y\vartheta}
{\sqrt J}\right\|_2^2+C_\eta(\|J\|_\infty^2+1),
$$
for any $\eta\in(0,\infty)$, where $C_\eta$ is a positive constant
depending only on $\eta$ and $N_1$, where
$$
N_1:=\frac{E_0}{c_v}+\bar\varrho+\frac{1}{\bar\varrho}+L+\frac1L+
\frac{1}{\omega_0}+\|\varrho_0'\|_\infty.
$$
\end{proposition}

\begin{proof}
(i) Choose $y_0\in(0,L)$ such that
$$
\varrho_0^2(y_0)\vartheta(y_0,t)\leq\frac2L\int_0^L\varrho_0^2\vartheta d\xi\leq\frac{2\bar\varrho}{L}\|\varrho_0\vartheta\|_1.
$$
By the H\"older and Young inequalities, we deduce
\begin{eqnarray*}
  &&(\varrho_0^2\vartheta)^2(y,t)\leq(\varrho_0\vartheta)^2(y_0,t)
  +2\int_0^L\varrho_0^2\vartheta|\partial_y(\varrho_0^2\vartheta)|d\xi\\
  &\leq&\left(\frac{2\bar\varrho}{L}\right)^2\|\varrho_0\vartheta\|_1^2
  +2\int_0^L(2\varrho_0^2\vartheta\varrho_0\vartheta
  |\varrho_0'|+\varrho_0^2\vartheta\varrho_0^2
  |\partial_y\vartheta|)d\xi\\
  &\leq&\left(\frac{2\bar\varrho}{L}\right)^2\|\varrho_0\vartheta\|_1^2
  +4\|\varrho_0^2\vartheta\|_\infty\|\varrho_0\vartheta\|_1
  \|\varrho_0'\|_\infty+2\bar\varrho^{\frac52}
  \left\|\frac{\partial_y\vartheta}{\sqrt
  J}\right\|_2\|J\|_\infty^{\frac12}\|\varrho_0\vartheta\|_1^{\frac12}
  \|\varrho_0^2\vartheta\|_\infty^{\frac12}\\
  &\leq&\frac12\|\varrho_0^2\vartheta\|_\infty^2
  +\left(\frac{2\bar\varrho}{L}\right)^2\|\varrho_0\vartheta\|_1^2
  +16\|\varrho_0\vartheta\|_1^2
  \|\varrho_0'\|_\infty^2+3\bar\varrho^{\frac{10}{3}}
  \|\varrho_0\vartheta\|_1^{\frac23}
  \left\|\frac{\partial_y\vartheta}{\sqrt
  J}\right\|_2^{\frac43}\|J\|_\infty^{\frac23},
\end{eqnarray*}
for any $y\in(0,L)$, and, thus,
\begin{equation*}
 \|\varrho_0^2\vartheta\|_\infty^2\leq
 2\left(\frac{2\bar\varrho}{L}\right)^2\|\varrho_0\vartheta\|_1^2
  +32\|\varrho_0\vartheta\|_1^2
  \|\varrho_0'\|_\infty^2+6\bar\varrho^{\frac{10}{3}}
  \|\varrho_0\vartheta\|_1^{\frac23}
  \left\|\frac{\partial_y\vartheta}{\sqrt
  J}\right\|_2^{\frac43}\|J\|_\infty^{\frac23},
\end{equation*}
from which, by Proposition \ref{PROPBASIC}, we have
$$
\|\varrho_0^2\vartheta\|_\infty^2\leq\left(\frac{E_0}{c_v}\right)^2
\left(\frac{8\bar\varrho^2}{L^2}+32\|\varrho_0'\|_\infty^2\right)
+6\bar\varrho^{\frac{10}{3}}\left(\frac{E_0}{c_v}\right)^{\frac23}
\left\|\frac{\partial_y\vartheta}
{\sqrt J}\right\|_2^{\frac43}\|J\|_\infty^{\frac23}.
$$

Noticing that
$$
\vartheta(y,t)=\frac{1}{\omega_0}\int_{\Omega_0}\vartheta dz+
\frac{1}{\omega_0}\int_{\Omega_0}\int_z^y\partial_y\vartheta d\xi dz,
$$
we deduce, by the H\"older inequality, that
\begin{eqnarray*}
\|\vartheta\|_\infty&\leq&\frac{1}{\omega_0}\int_{\Omega_0}
\frac{\varrho_0\vartheta}{\varrho_0}dz+\int_0^L
|\partial_y\vartheta|dz\\
&\leq&\frac{2}{\omega_0\bar\varrho}\|\varrho_0\vartheta\|_1+\left(
\int_0^L\left|\frac{\partial_y\vartheta}{\sqrt J}\right|^2
dz\right)^{\frac12}\left(\int_0^LJdz\right)^{\frac12},
\end{eqnarray*}
from which, by Proposition \ref{PROPBASIC}, one obtains
$$
\|\vartheta\|_\infty\leq \sqrt L
\left\|\frac{\partial_y\vartheta}
{\sqrt J}\right\|_2+\frac{2E_0}{c_v\omega_0\bar\varrho}.
$$

(ii) Thanks to (i), one has
\begin{eqnarray*}
\|\varrho_0^2\vartheta\|_\infty^2\leq C\left(1+
\left\|\frac{\partial_y\vartheta}
{\sqrt J}\right\|_2^{\frac43}\|J\|_\infty^{\frac23}\right),\quad
\|\vartheta\|_\infty\leq C\left(
\left\|\frac{\partial_y\vartheta}
{\sqrt J}\right\|_2+1\right),
\end{eqnarray*}
for a positive constant $C$ depending only on $N_1$. Therefore, we have
\begin{eqnarray*}
  \|\sqrt{\varrho_0}\vartheta\|_\infty^2&=&\left\|(\varrho_0^2\vartheta)
  ^{\frac14}\vartheta^{\frac34}\right\|_\infty^2\leq\|\varrho_0^2\vartheta
  \|_\infty^{\frac12}\|\vartheta\|_\infty^{\frac32}\\
  &\leq&C\left(1+
  \left\|\frac{\partial_y\vartheta}
  {\sqrt J}\right\|_2^{\frac43}\|J\|_\infty^{\frac23}\right)^{\frac14}
  \left(\left\|\frac{\partial_y\vartheta}
  {\sqrt J}\right\|_2+1\right)^{\frac32},
\end{eqnarray*}
for a positive constant $C$ depending only on $N_1$, from which, by the
Young inequality, (ii) follows.
\end{proof}

We can now prove the desired a priori
$L^\infty(0,T; L^2)\cap L^2(0,T; H^1)$ estimates on $(v,\vartheta)$ as
in the next proposition.

\begin{proposition}
\label{PROPEstL2}
Given $T\in(0,\infty)$. It holds that
\begin{align*}
\sup_{0\leq t\leq T}\left(\|(\sqrt{\varrho_0}v^2,\sqrt{\varrho_0}\vartheta)\|_2^2 +\|J\|_\infty^2\right)&+\int_0^T\left(\|\vartheta\|_\infty^2+\|(\partial_y
\vartheta,v\partial_y v)\|_2^2\right)dt\\
\leq&\ \ C(1+\|(\sqrt{\varrho_0}
\vartheta_0, \sqrt{\varrho_0}v_0^2)\|_2^2),
\end{align*}
for a positive constant $C$ depending only on $R, c_v, \mu, \kappa,
m_1, N_1$, and $T$, where $m_1$ and $N_1$ are the numbers in
Proposition \ref{PROPEstJ} and Proposition \ref{PROPEstRhoWei}, respectively. 
\end{proposition}

\begin{proof}
Denote $\mathcal E:=\frac{v^2}{2}+c_v\vartheta$. Then one can derive
from (\ref{EQv}) and (\ref{EQtheta}) that
\begin{equation}
  \label{EQE}
  \varrho_0\partial_t\mathcal E+\partial_y(v\pi)-\kappa \partial_y\left(
  \frac{\partial_y\vartheta}{J}\right)=\mu\partial_y\left(\frac1J
  \partial_y\left(\frac{v^2}{2}\right)\right).
\end{equation}
Multiplying (\ref{EQE}) by $\mathcal E$ and integrating the resultant
over $(0,L)$, one get from integration by parts that
\begin{equation}
  \frac12\frac{d}{dt}\int_0^L\varrho_0\mathcal E^2 dy+\int_0^L\frac1J(
  \kappa\partial_y\vartheta+\mu v\partial_y v)\partial_y\mathcal E dy=
  \int_0^Lv\pi\partial_y\mathcal E dy. \label{L2-1}
\end{equation}
By the Young inequality, we have
\begin{eqnarray*}
  \int_0^L\frac1J(
  \kappa\partial_y\vartheta+\mu v\partial_y v)\partial_y\mathcal E dy
  &=&\int_0^L\frac1J(
  \kappa\partial_y\vartheta+\mu v\partial_y v)
  (v\partial_y v+ c_v\partial_y\vartheta) dy\\
  &\geq&\frac{3\kappa
  c_v}{4}\int_0^L\left|\frac{\partial_y\vartheta}{\sqrt J}
  \right|^2dy-C\int_0^L\left|\frac{v\partial_yv}{\sqrt J}\right|^2dy
\end{eqnarray*}
and
\begin{eqnarray*}
  \int_0^Lv\pi\partial_y\mathcal Edy&=&R\int_0^Lv\frac{\varrho_0}{J}
  \vartheta(v\partial_yv+c_v\partial_y\vartheta) dy\\
  &\leq&\frac{\kappa c_v}{4}\int_0^L\left|\frac{\partial_y\vartheta}
  {\sqrt{J}}\right|^2dy +C\int_0^L\frac1J[\varrho_0^2v^2\vartheta^2+(v
  \partial_yv)^2]dy,
\end{eqnarray*}
for a positive constant $C$ depending only on $R, c_v, \mu$, and
$\kappa$. Substituting the above two inequalities into (\ref{L2-1})
and applying Proposition \ref{PROPBASIC} and Proposition \ref{PROPEstJ},
one obtains
\begin{eqnarray*}
  \frac{d}{dt}\|\sqrt{\varrho_0}\mathcal E\|_2^2+\kappa c_v\left\|
  \frac{\partial_y\vartheta}{\sqrt J}\right\|_2^2 \leq
  C\left(\left\|\frac{v\partial_yv}{\sqrt J}\right\|_2^2+\int_0^L\frac{\varrho_0v^2}{J}\varrho_0\vartheta^2
  dy\right)\\
  \leq C\left(\left\|\frac{v\partial_yv}{\sqrt J}\right\|_2^2+
  \|\sqrt{\varrho_0}v\|_2^2\|\sqrt{\varrho_0}\vartheta\|_\infty^2
  m_1f_1(t)\right)\\
  \leq C\left(\left\|\frac{v\partial_yv}{\sqrt J}\right\|_2^2+
  E_0m_1f_1(t)\|\sqrt{\varrho_0}\vartheta\|_\infty^2
  \right),
\end{eqnarray*}
for a positive constant $C$ depending only on $R, c_v, \mu$, and
$\kappa$, and, thus,
\begin{equation}
\label{L2-EINQE}
  \frac{d}{dt}\|\sqrt{\varrho_0}\mathcal E\|_2^2+\kappa c_v\left\|
  \frac{\partial_y\vartheta}{\sqrt J}\right\|_2^2
  \leq A_1\left(\left\|\frac{v\partial_yv}{\sqrt J}\right\|_2^2+
  E_0m_1f_1(t)\|\sqrt{\varrho_0}\vartheta\|_\infty^2
  \right),
\end{equation}
for a positive constant $A_1$ depending only on $R, c_v, \mu$, and
$\kappa$.

Multiplying (\ref{EQv}) by $4v^3$ and integrating the resultant over
$(0,L)$, one gets from integration by parts and the Young inequality
that
\begin{eqnarray*}
  &&\frac{d}{dt}\int_0^L\varrho_0 v^4
  dy+12\mu\int_0^L\left|\frac{v\partial_yv}{\sqrt J}\right|^2dy\\
  &=&12\int_0^L\pi v^2\partial_yvdy=12R\int_0^L\frac{\varrho_0}{J}
  \vartheta v^2\partial_yvdy \\
  &\leq&6\mu\int_0^L\left|\frac{v\partial_yv}{\sqrt J}\right|^2 dy +
  \frac{6R^2}{\mu}\int_0^L\frac{\varrho_0v^2}{J}\varrho_0\vartheta^2
  dy,
\end{eqnarray*}
from which, by Proposition \ref{PROPEstJ}, the H\"older inequality, and
Proposition \ref{PROPBASIC}, one obtains
\begin{eqnarray*}
  &&\frac{d}{dt}\int_0^L\varrho_0 v^4
  dy+6\mu\int_0^L\left|\frac{v\partial_yv}{\sqrt J}\right|^2dy
  \leq\frac{6R^2}{\mu}\int_0^L\frac{\varrho_0v^2}{J}\varrho_0\vartheta^2
  dy\\
  &\leq&\frac{6R^2}{\mu}m_1f_1(t)\|\sqrt{\varrho_0}v\|_2^2\|\sqrt{
  \varrho_0}\vartheta\|_\infty^2
  \leq\frac{12R^2m_1E_0}{\mu}f_1(t)\|\sqrt{\varrho_0}
  \vartheta\|_\infty^2,
\end{eqnarray*}
that is,
\begin{equation}
  \label{L2-EINQv}
  \frac{d}{dt}\|\sqrt{\varrho_0} v^2\|_2^2
  +6\mu\left\|\frac{v\partial_yv}{\sqrt J}\right\|_2^2
  \leq\frac{12R^2}{\mu}E_0 m_1 f_1(t)\|\sqrt{\varrho_0}
  \vartheta\|_\infty^2.
\end{equation}

Multiplying (\ref{L2-EINQv}) by $\frac{A_1}{\mu}$, adding the resultant
to (\ref{L2-EINQE}), and noticing that $f_1(t)$ is nondecreasing in $t$,
one gets
\begin{eqnarray}
  &&\frac{d}{dt}\left(\|\sqrt{\varrho_0}\mathcal E\|_2^2
  +\frac{A_1}{\mu}\|\sqrt{\varrho_0}v^2\|_2^2\right)
  +5A_1\left\|\frac{v\partial_yv}{\sqrt J}\right\|_2^2+\kappa c_v
  \left\|\frac{\partial_y\vartheta}{\sqrt{J}}\right\|_2^2 \nonumber\\
  &\leq&\left(1+\frac{12R^2}{\mu^2}\right)A_1E_0m_1f_1(t)\|\sqrt{\varrho_0}
  \vartheta\|_\infty^2
  \leq  A_2E_0m_1f_1(T)\|\sqrt{\varrho_0}
  \vartheta\|_\infty^2(t), \label{L2-EINQEv}
\end{eqnarray}
for any $t\in(0,T)$, where $ A_2=\left(1+\frac{12R^2}
{\mu^2}\right)A_1$.

By Proposition \ref{PROPEstJ} and (ii) of
Proposition \ref{PROPEstRhoWei}, we have
\begin{equation*}
    \|\sqrt{\varrho_0}\vartheta\|_\infty^2(t)\leq\eta\left\|
    \frac{\partial_y\vartheta}{\sqrt
     J}\right\|_2^2(t)+C_\eta\left(1+\int_0^t\|\sqrt{\varrho_0}
     \vartheta\|_\infty^2 d\tau\right),
\end{equation*}
for any $\eta\in(0,\infty)$, and any $t\in(0,T)$,
where $C_\eta$ is a positive constant
depending only on $R, c_v, \mu, \kappa,$ $m_1, N_1, T$, and $\eta$.
Multiplying both sides of the above inequality by $2
A_2E_0m_1f_1(T)$, choosing $\eta=\frac{\kappa c_v}{4 A_2E_0m_1f_1(T)}$, and summing the resultant with (\ref{L2-EINQEv}),
one obtains
\begin{align*}
\frac{d}{dt}&\left(\|\sqrt{\varrho_0}\mathcal E\|_2^2
  +\frac{A_1}{\mu}\|\sqrt{\varrho_0}v^2\|_2^2
  +A_2E_0m_1f_1(T)\int_0^t\|\sqrt{\varrho_0}\vartheta\|_\infty^2
  d\tau\right)\\
  &+5A_1\left\|\frac{v\partial_yv}{\sqrt J}\right\|_2^2
  +\frac{\kappa c_v}{2}
  \left\|\frac{\partial_y\vartheta}{\sqrt{J}}\right\|_2^2
  \leq CA_2E_0m_1f_1(T)
  \left(1+ \int_0^t\|\sqrt{\varrho_0}
     \vartheta\|_\infty^2 d\tau\right),
\end{align*}
for any $t\in(0,T)$, where $C$ is a positive constant depending only
on $R, c_v, \mu, \kappa, m_1, N_1$, and $T$. Applying the Gronwall inequality to the above inequality, one gets
\begin{align}
\sup_{0\leq t\leq T} \|(\sqrt{\varrho_0}v^2,\sqrt{\varrho_0}\vartheta)\|_2^2& +\int_0^T (\|\sqrt{\varrho_0}\vartheta\|_\infty^2+\|(\partial_y
\vartheta,v\partial_y v)\|_2^2) dt\nonumber \\
\leq&\ \  C(1+\|(\sqrt{\varrho_0}
\vartheta_0, \sqrt{\varrho_0}v_0^2)\|_2^2),\label{L2-EstEv}
\end{align}
for a positive constant $C$ depending only on $R, c_v, \mu, \kappa,
m_1, N_1$, and $T$. The desired estimate
$$
\sup_{0\leq t\leq T}\|J\|_\infty^2+\int_0^T\|\vartheta\|_\infty^2 dt\leq C(1+\|(\sqrt{\varrho_0}
\vartheta_0, \sqrt{\varrho_0}v_0^2)\|_2^2)
$$
follows from (\ref{L2-EstEv}),
by applying Proposition \ref{PROPEstJ} and
(i) of Proposition \ref{PROPEstRhoWei}.
\end{proof}

As a corollary of Proposition \ref{PROPEstJ} and
Proposition \ref{PROPEstL2}, we have the following:

\begin{corollary}
  \label{COREstL2}
Given $T\in(0,\infty)$. It holds that
$$
0<\underline{\mathcal J}
\leq J(y,t)\leq  C(1+\|(\sqrt{\varrho_0}v_0, \sqrt{\varrho_0}v_0^2,\sqrt{\varrho_0}
\vartheta_0)\|_2^2),
$$
for all $(y,t)\in(0,L)\times(0,T),$ and
\begin{eqnarray*}
  \sup_{0\leq t\leq T}\|(\sqrt{\varrho_0}v,\sqrt{\varrho_0}\vartheta)\|_2^2
  +\int_0^T\left(\|\vartheta\|_\infty^2+\|(\partial_y
  \vartheta,\partial_y v)\|_2^2\right)dt\\
  \leq C(1+\|(\sqrt{\varrho_0}v_0, \sqrt{\varrho_0}v_0^2,\sqrt{\varrho_0}
  \vartheta_0)\|_2^2),
\end{eqnarray*}
for positive constants $\underline{\mathcal J}$ and
$C$ depending only on $R, c_v, \mu, \kappa,
m_1, N_1$, and $T$, where $m_1$ and $N_1$ are the numbers in
Proposition \ref{PROPEstJ} and Proposition \ref{PROPEstRhoWei}, respectively.
\end{corollary}

\begin{proof}
  All estimates expect that for $\int_0^T\|\partial_yv\|_2^2dt$
  follow directly from Proposition \ref{PROPEstJ} and Proposition
  \ref{PROPEstL2}.
  Multiplying (\ref{EQv}) by $v$, integrating the resultant over $(0,L)$, one gets from integration by parts that
  \begin{eqnarray*}
    &&\frac12\frac{d}{dt}\|\sqrt{\varrho_0}v\|_2^2+\mu\left\|
    \frac{\partial_y v}{\sqrt J}\right\|_2^2
    =\int_0^L\pi\partial_yvdy\\
    &=&R\int_0^L\frac{\varrho_0}{J}\vartheta\partial_yvdy
    \leq\frac\mu2
    \left\|\frac{\partial_y v}{\sqrt J}\right\|_2^2+\frac{R}{2\mu}\int_0^L\frac{\varrho_0^2}{J}\vartheta^2 dy
  \end{eqnarray*}
  and, thus, by Proposition \ref{PROPEstJ},
  $$
  \frac{d}{dt}\|\sqrt{\varrho_0}v\|_2^2+\mu\left\|\frac{\partial_y v}{\sqrt J}\right\|_2^2
  \leq\frac{R}{\mu}m_1f_1(t)\bar\varrho \|\sqrt{\varrho_0}\vartheta\|_\infty^2,
  $$
  from which, by Proposition \ref{PROPEstL2}, the conclusion follows.
\end{proof}

\subsection{A priori $H^1$ estimates} This section is devoted to the a
prior $H^1$ type estimates on $(J, v,\vartheta)$.
Precisely, we will carry out the
a priori $L^\infty(0,T; H^1)\cap L^2(0,T; H^2)$ estimate on $v$
and the $L^\infty(0,T; H^1)$ estimate on $J$; however, due to
the presence of
the term $\frac{\mu}{J}(\partial_yv)^2$ on the right-hand side of the
equation for $\vartheta$, (\ref{EQtheta}), one can not get the desired
a priori $H^1$ estimate of $\vartheta$ independent of the lower bound of
the density, without appealing to the higher than $H^1$ energy
estimates.

Define the effective viscous flux $G$ as
$$
G:=\mu\frac{\partial_yv}{J}-\pi=\mu\frac{\partial_yv}{J}-
R\frac{\varrho_0}{J}\vartheta.
$$
Then, one can derive from (\ref{EQJ})--(\ref{EQtheta}) that
\begin{equation}
  \label{EQG}
  \partial_tG-\frac\mu J\partial_y\left(\frac{\partial_yG}{\varrho_0}\right)
  =-\frac{\kappa(\frac{R}{c_v}-1)}{J}\partial_y\left(
  \frac{\partial_y\vartheta}{J}
  \right)-\frac{R}{c_v}\frac{\partial_yv}{J}G.
\end{equation}
Moreover, by equation (\ref{EQv}), one has  $\partial_yG=\varrho_0\partial_tv$, from which, recalling the boundary
condition (\ref{BC}), we have
$$
\partial_yG(0,t)=\partial_yG(1,t)=0,\quad t\in(0,\infty).
$$

We have the a priori $L^2$ estimates on $G$ sated in the following:

\begin{proposition}
  \label{PROPEstGL2}
Given $T\in(0,\infty)$. It holds that
\begin{equation*}
  \sup_{0\leq t\leq T}\|G\|_2^2+\int_0^T\left\|\frac{\partial_y G}{\sqrt{\varrho_0}}\right\|_2^2dt\leq C,
\end{equation*}
for a positive constant $C$ depending only on $R, c_v, \mu, \kappa, m_1, N_1$, $N_2$, and $T$, where
$$
N_2:=\|\sqrt{\varrho_0}v_0^2\|_2+
\|\sqrt{\varrho_0}\vartheta_0\|_2+\|v_0'\|_2,
$$
and $m_1$ and $N_1$ are the numbers in
Proposition \ref{PROPEstJ} and Proposition \ref{PROPEstRhoWei}, respectively.
\end{proposition}

\begin{proof}
Multiplying equation (\ref{EQG}) by $JG$,
integrating the resultant over
$(0,L)$, and recalling $\partial_yG|_{y=0,L}=0$, one gets from integration by parts that
\begin{equation*}
  \int_0^L\partial_tGJGdy+\mu\int_0^L\left|\frac{\partial_yG}{\sqrt J}\right|^2dy=\kappa\left(\frac{R}{c_v}-1\right)
  \int_0^L\frac{\partial_y\vartheta
  \partial_yG}{J}dy-\frac{R}{c_v}\int_0^L\partial_yvG^2dy.
\end{equation*}
Using (\ref{EQJ}), one has
\begin{eqnarray*}
  \int_0^L\partial_tGJGdy&=&\frac12\frac{d}{dt}\int_0^LJG^2dy
  -\frac12\int_0^L
  \partial_tJG^2dy\\
  &=&\frac12\frac{d}{dt}\int_0^LJG^2dy-\frac12\int_0^L\partial_yvG^2dy.
\end{eqnarray*}
Therefore, it follows from the H\"older, Young, and Gagliardo-Nirenberg  inequalities and Corollary \ref{COREstL2} that
\begin{eqnarray*}
  &&\frac12\frac{d}{dt}\int_0^LJG^2dy+ \mu\int_0^L\left|\frac{\partial_yG}{\sqrt J}\right|^2dy\\
  &=&\kappa\left(\frac{R}{c_v}-1\right)\int_0^L\frac{\partial_y\vartheta
  \partial_yG}{J}dy+\left(\frac12-\frac{R}{c_v}
  \right)\int_0^L\partial_yvG^2dy\\
  &\leq&\kappa\left|\frac{R}{c_v}-1\right|\sqrt{\bar\varrho}
  \left\|\frac{\partial_yG}{\sqrt{\varrho_0}}\right\|_2\left\|
  \frac{\partial_y\vartheta}{\sqrt J}\right\|_2+\left(\frac12+\frac{R}{c_v}\right)\|\partial_yv\|_2\|G\|_2\|G\|_\infty\\
  &\leq&C\left(\left\|\frac{\partial_yG}{\sqrt{\varrho_0}}\right\|_2
  \|\partial_y\vartheta\|_2+\|\partial_yv\|_2\|G\|_2^{\frac32}(\|G\|_2
  +\|\partial_yG\|_2)^{\frac12}\right)\\
  &\leq&\frac\mu2\left\|\frac{\partial_yG}{\sqrt{\varrho_0}}\right\|_2^2
  +C\Big[(1+\|\partial_yv\|_2^2)\|G\|_2^2+\|\partial_y\vartheta\|_2^2\Big],
\end{eqnarray*}
that is
\begin{equation}
\label{EINQGL2}
  \frac{d}{dt}\|\sqrt JG\|_2^2+\mu\left\|\frac{\partial_yG}{\sqrt J}
  \right\|_2^2\leq C\Big[(1+\|\partial_yv\|_2^2)\|G\|_2^2+\|\partial_y\vartheta\|_2^2\Big],
\end{equation}
for any $t\in(0,T)$, where $C$ is a positive constant depending only on
$R, c_v, \mu, \kappa, m_1, N_1$, and $T$. Applying the Gronwall
inequality to (\ref{EINQGL2}) and using Corollary \ref{COREstL2}, the conclusion follows.
\end{proof}

Based on Proposition \ref{PROPEstGL2}, we can obtain the
desired $H^1$ type estimates on $J$ and $v$ as stated in the next proposition.

\begin{proposition}
  \label{PROPEstJvH1}
Given $T\in(0,\infty)$. It holds that
\begin{equation*}
  \sup_{0\leq t\leq T}(\|\partial_yJ\|_2^2+\|\partial_yv\|_2^2)
  +\int_0^T(\|\sqrt{\varrho_0}\partial_tv\|_2^2+\|\partial_y^2v\|_2^2)dt
  \leq C,
\end{equation*}
for a positive constant $C$ depending only on
$R, c_v, \mu, \kappa, m_1, N_1, N_2$, and $T$, where $m_1$, $N_1$ and $N_2$ are the numbers in
Propositions \ref{PROPEstJ}, \ref{PROPEstRhoWei}, and \ref{PROPEstGL2}, respectively.
\end{proposition}

\begin{proof}
The estimate
$\sup_{0\leq t\leq T}\|\partial_yv\|_2^2
  +\int_0^T\|\sqrt{\varrho_0}\partial_tv\|_2^2dt\leq C$ is
straightforward from Corollary \ref{COREstL2} and Proposition \ref{PROPEstGL2}, by
the definition of $G$ and noticing that
$\varrho_0\partial_tv=\partial_yG$.
Note that, by the Sobolev embedding inequality, it follows from Proposition \ref{PROPEstGL2} that
\begin{equation}
  \label{EstGInf}
  \int_0^T\|G\|_\infty^2dt\leq C\int_0^L\|G\|_{H^1}^2dt\leq C,
\end{equation}
for a positive constant $C$ depending only on $R, c_v, \mu, \kappa, m_1, N_1, N_2$, and $T$.

  Rewrite (\ref{EQJ}) in terms of $G$ as
  $$
  \partial_tJ=\frac1\mu(JG+R\varrho_0\vartheta).
  $$
  Differentiating the above equations in $y$, multiplying the resultant
  by $\partial_yJ$, and integrating over $(0,L)$, it follows from the
  H\"older and Young inequalities that
  \begin{eqnarray*}
    &&\frac12\frac{d}{dt}\|\partial_yJ\|_2^2=\frac1\mu\int_0^L
    [G|\partial_y
    J|^2+\partial_yGJ\partial_yJ+R(\varrho_0'\vartheta+\varrho_0
    \partial_y
    \vartheta)\partial_yJ]dy \\
    &\leq&\frac1\mu\left[\|G\|_\infty\|\partial_yJ\|_2^2+\|J\|_\infty
    \|\partial_yG\|_2\|\partial_yJ\|_2+R(\|\vartheta\|_\infty
    \|\varrho_0'\|_2+\bar\varrho\|\partial_y\vartheta\|_2)
    \|\partial_yJ\|_2
    \right]\\
    &\leq&C(\|G\|_\infty^2+1)\|\partial_yJ\|_2^2+C(\|J\|_\infty^2
    \|\partial_y G\|_2^2+\|\varrho_0'\|_2^2\|\vartheta\|_\infty^2+\bar\varrho^2 \|\partial_y\vartheta\|_2^2),
  \end{eqnarray*}
  for a positive constant $C$ depending only on $R$ and $\mu$.
  Applying the Gronwall inequality to the above inequality, it follows
  from (\ref{EstGInf}), Corollary \ref{COREstL2}, and Proposition
  \ref{PROPEstGL2} that
  \begin{equation}
  \label{EstJH1}
    \sup_{0\leq t\leq T}\|\partial_yJ\|_2^2\leq C,
  \end{equation}
  for a positive constant $C$ depending only on $R, c_v, \mu, \kappa, m_1, N_1, N_2$, and $T$.

Noticing that
$\partial_yv=\frac1\mu(JG+R\varrho_0\vartheta)$,
one has
\begin{eqnarray*}
\partial_y^2v=\frac1\mu(\partial_yJG+J\partial_yG+R\varrho_0'\vartheta
+R\varrho_0\partial_y\vartheta),
\end{eqnarray*}
and, thus, by the H\"older inequality, (\ref{EstGInf}), (\ref{EstJH1}), it follows from Corollary \ref{COREstL2} and Proposition \ref{PROPEstGL2} that
  \begin{eqnarray*}
  \int_0^T\|\partial_y^2v\|_2^2dt
  &\leq& C\int_0^T(\|\partial_yJ\|_2^2\|G\|_\infty^2+
  \|J\|_\infty^2\|\partial_yG\|_2^2+\|\varrho_0'\|_2^2
  \|\vartheta\|_\infty^2+\|\partial_y\vartheta\|_2^2)dt\\
  &\leq& C\int_0^T(\|G\|_\infty^2+
  \|J\|_\infty^2\|\partial_yG\|_2^2+
  \|\vartheta\|_\infty^2+\|\partial_y\vartheta\|_2^2)dt\leq C,
  \end{eqnarray*}
  for a positive constant $C$ depending only on $R, c_v, \mu, \kappa, m_1, N_1, N_2$, and $T$, proving the conclusion.
\end{proof}

We summarize the estimates obtained in this subsection in the following:

\begin{corollary}
  \label{COREstH1}
  Given $T\in(0,\infty)$. It holds that
  \begin{eqnarray*}
    \sup_{0\leq t\leq T}\|(G,\partial_yJ,\partial_yv)\|_2^2+\int_0^T
    \left\|\left(\frac{\partial_yG}{\sqrt{\varrho_0}},\partial_y^2v,
    \sqrt{\varrho_0}\partial_tv\right)\right\|_2^2 dt\leq C,
  \end{eqnarray*}
  for a positive constant $C$ depending only on $R, c_v, \mu, \kappa, m_1, N_1, N_2$, and $T$, where $m_1$, $N_1$ and $N_2$ are the numbers in Propositions \ref{PROPEstJ}, \ref{PROPEstRhoWei}, and \ref{PROPEstGL2}, respectively.
\end{corollary}

\subsection{A priori $H^2$ estimates}
This subsection is devoted to the a prior $H^2$ estimates on $(J, v,
\vartheta)$. As will be shown in this
subsection that one can get the desired
a priori $L^\infty(0,T; H^2)$ estimate of $\vartheta$,
without using the a priori $L^\infty(0,T; H^1)$ bound of it.

As a preparation, we first give some estimates on
$\|\partial_y\vartheta\|_2$ and $\|\partial_t\vartheta\|_\infty$,
in terms of $\|\sqrt{\varrho_0}\partial_t\vartheta\|_2$
and $\|\partial_y\partial_t\vartheta\|_2$, stated in the following
proposition.

\begin{proposition}
 \label{PROPEstH2-PRE}
Given $T\in(0,\infty)$.

(i) It holds that
 \begin{eqnarray*}
   \|\partial_y\vartheta\|_2^2\leq C(1+\|\sqrt{\varrho_0}
  \partial_t\vartheta\|_2),
  \end{eqnarray*}
for a positive constant $C$ depending only on
$R, c_v, \mu, \kappa, m_1, N_1, N_2$, and $T$, where 
$m_1$, $N_1$ and $N_2$ are the numbers in
Propositions \ref{PROPEstJ}, \ref{PROPEstRhoWei}, and \ref{PROPEstGL2}, respectively.

(ii) It holds that
  \begin{eqnarray*}
  \|\partial_t\vartheta\|_\infty&\leq& \sqrt{\frac{2}{\omega_0\bar\varrho}}\|\sqrt{\varrho_0}\partial_t
  \vartheta\|_2 +\sqrt L\left\|\frac{\partial_y\partial_t\vartheta}{
  \sqrt J}\right\|_2, \\
  \|\partial_tv\|_\infty&\leq& \sqrt{\frac{2}{\omega_0\bar\varrho}}\|\sqrt{\varrho_0}\partial_t
  v\|_2 +\sqrt L\left\|\frac{\partial_y\partial_tv}{
  \sqrt J}\right\|_2,
 \end{eqnarray*}
 where $\omega_0$ is the number in Proposition \ref{PROPEstRhoWei}.
\end{proposition}

\begin{proof}
(i)
Multiplying (\ref{EQtheta'}) by $\vartheta$, integrating the resultant
over $(0,L)$, and integrating by parts, it follows from the H\"older
inequality that
\begin{eqnarray*}
  \kappa\int_0^L\left|\frac{\partial_y\vartheta}{\sqrt J}\right|^2 dy
  &=&\int_0^L(\partial_yvG-c_v\varrho_0\partial_t\vartheta)\vartheta dy
  \\
  &\leq&\|\partial_yv\|_2\|G\|_2\|\vartheta\|_\infty+c_v\|\sqrt{\varrho_0}
\partial_t\vartheta\|_2\|\sqrt{\varrho_0}\vartheta\|_2,
\end{eqnarray*}
from which, by Corollaries \ref{COREstL2}--\ref{COREstH1} and (i) of Proposition \ref{PROPEstRhoWei} that
\begin{eqnarray*}
  &&\|\partial_y\vartheta\|_2^2\leq C \int_0^L\left|\frac{\partial_y\vartheta}{\sqrt J}\right|^2 dy
  \leq C(\|\vartheta\|_\infty+\|\sqrt{\varrho_0}
  \partial_t\vartheta\|_2) \\
  &\leq&C(\|\partial_y\vartheta\|_2+1+\|\sqrt{\varrho_0}
  \partial_t\vartheta\|_2)\leq \frac12 \|\partial_y\vartheta\|_2^2
  +C(1+\|\sqrt{\varrho_0}
  \partial_t\vartheta\|_2)
\end{eqnarray*}
and, thus,
\begin{equation*}
\label{PRE1}
  \|\partial_y\vartheta\|_2^2\leq C(1+\|\sqrt{\varrho_0}
  \partial_t\vartheta\|_2),
\end{equation*}
for a positive constant $C$ depending only on
$R, c_v, \mu, \kappa, m_1, N_1, N_2$, and $T$.

(ii) Recall that
$\Omega_0:=\left\{y\in(0,L)
\big|\varrho_0(y)\geq\frac{\bar\varrho}{2}\right\}$ and $|\omega_0|=|\Omega_0|>0.$
Noticing
$$
\partial_t\vartheta(y,t)=\frac{1}{\omega_0} \int_{\Omega_0}
 \partial_t\vartheta(z,t) dz
+\frac{1}{\omega_0}\int_{\Omega_0}\int_z^y\partial_y\partial_t\vartheta
(\xi,t)d\xi dz,
$$
it follows from the H\"older inequality and Proposition \ref{PROPBASIC}
that
\begin{eqnarray*}
  |\partial_t\vartheta(y,t)|&\leq&\frac{1}{\omega_0} \left|
  \int_{\Omega_0}
  \frac{\sqrt{\varrho_0}\partial_t\vartheta}{\sqrt{\varrho_0}}dz\right|
  +\int_0^L|\partial_y\partial_t\vartheta(\xi,t)|d\xi \\
  &\leq&\sqrt{\frac{2}{\omega_0\bar\varrho}}\|\sqrt{\varrho_0}\partial_t
  \vartheta\|_2 +\left(\int_0^L\left|\frac{\partial_y\partial_t\vartheta}{\sqrt J}
  \right|^2d\xi\right)^{\frac12}
  \left(\int_0^LJd\xi\right)^{\frac12}\\
  &=&\sqrt{\frac{2}{\omega_0\bar\varrho}}\|\sqrt{\varrho_0}\partial_t
  \vartheta\|_2 +\sqrt L\left\|\frac{\partial_y\partial_t\vartheta}{
  \sqrt J}\right\|_2,
\end{eqnarray*}
which implies
\begin{equation*}
\label{PRE2}
\|\partial_t\vartheta\|_\infty\leq \sqrt{\frac{2}{\omega_0\bar\varrho}}\|\sqrt{\varrho_0}\partial_t
  \vartheta\|_2 +\sqrt L\left\|\frac{\partial_y\partial_t\vartheta}{
  \sqrt J}\right\|_2.
\end{equation*}
In the same way as above, the same conclusion holds for $\partial_tv$.
\end{proof}

\begin{proposition}
  \label{PROPEstH2-A}
Given $T\in(0,\infty)$. It holds that
\begin{equation*}
  \sup_{0\leq t\leq T}\left\|\left(\sqrt{\varrho_0}\partial_t\vartheta,
  \frac{\partial_yG}{\sqrt{\varrho_0}}\right)\right\|_2^2
  +\int_0^T\|(\partial_tG,\partial_y\partial_t\vartheta)\|_2^2
  dt\leq C(\|g_0\|_2+\|h_0\|_2),
\end{equation*}
for a positive constant $C$ depending only on
$R, c_v, \mu, \kappa, m_1, N_1, N_2$, and $T$, where
$$
g_0:=\frac{\mu v_0''-R(\varrho_0\theta_0)'}{\sqrt{\varrho_0}},
\quad h_0:=\frac{1}{\sqrt{\varrho_0}}\left[\mu(v_0')^2+\kappa\vartheta_0''
-Rv_0'\varrho_0\vartheta_0\right],
$$
and and $m_1, N_1$ and $N_2$ are the numbers in
  Propositions \ref{PROPEstJ}, \ref{PROPEstRhoWei}, and \ref{PROPEstGL2}, respectively.
\end{proposition}

\begin{proof}
Rewrite (\ref{EQtheta}) as
\begin{equation}
\label{EQtheta'}
c_v\varrho_0\partial_t\vartheta-\kappa\partial_y\left(\frac{\partial_y
\vartheta}{J}\right)=\partial_yvG,
\end{equation}
or, equivalently,
$$
c_v\varrho_0\partial_t\vartheta-\kappa\partial_y\left(\frac{\partial_y
\vartheta}{J}\right)=\frac1\mu(JG+R\varrho_0\vartheta)G,
$$
from which, differentiating in $t$ and using (\ref{EQJ}), one has
\begin{eqnarray*}
&&c_v\varrho_0\partial_t^2\vartheta-\kappa\partial_y
\left(\frac{\partial_y\partial_t\vartheta}{J}
-\frac{\partial_yv\partial_y\vartheta}{J^2}\right)\\
&=&\frac1\mu(\partial_yvG^2+2JG\partial_tG)
+\frac {R\varrho_0}{\mu}(\partial_t
\vartheta G+\vartheta\partial_tG)\\
&=&\frac{\partial_yv}{\mu}G^2+\frac1\mu(2JG+R\varrho_0\vartheta)
\partial_tG+\frac R\mu\varrho_0\partial_t\vartheta G.
\end{eqnarray*}
Multiplying the above equation by $\partial_t\vartheta$, integrating the
resultant over $(0,L)$, one gets from integration by parts that
\begin{eqnarray}
  &&\frac{c_v}{2}\frac{d}{dt}\int_0^L\varrho_0|\partial_t\vartheta|^2dy
  +\kappa\int_0^L\left|\frac{\partial_y\partial_t\vartheta}{\sqrt J}
  \right|^2 dy\nonumber\\
  &=&\kappa\int_0^L\frac{\partial_yv\partial_y\vartheta}{J^2}
  \partial_y\partial_t\vartheta dy +\frac R\mu\int_0^L\varrho_0G(\partial_t\vartheta)^2 dy\nonumber \\
  &&+\frac1\mu\int_0^L[(2JG+R\varrho_0
  \vartheta)\partial_tG+\partial_yvG^2]\partial_t\vartheta dy.\label{EINQGH1}
\end{eqnarray}

The terms on the right-hand side of (\ref{EINQGH1})
are estimated as follows. By Corollary \ref{COREstL2}, it follows from
the Young inequality and (i) of Proposition \ref{PROPEstH2-PRE} that
\begin{eqnarray*}
  \kappa\int_0^L\frac{\partial_yv\partial_y\vartheta}{J^2}
  \partial_y\partial_t\vartheta dy
  &\leq&\frac\kappa4\left\|\frac{\partial_y\partial_t\vartheta}{\sqrt
  J}\right\|_2^2+C\|\partial_yv\|_\infty^2\|
  \partial_y\vartheta\|_2^2\\
  &\leq&\frac\kappa4\left\|\frac{\partial_y\partial_t\vartheta}{\sqrt
  J}\right\|_2^2+C(\|G\|_\infty^2+\|\vartheta\|_\infty^2)
  \|\partial_y\vartheta\|_2^2\\
  &\leq&\frac\kappa4\left\|\frac{\partial_y\partial_t\vartheta}{\sqrt
  J}\right\|_2^2+C(\|G\|_\infty^2+\|\vartheta\|_\infty^2)
  (1+\|\sqrt{\varrho_0}
  \partial_t\vartheta\|_2),
\end{eqnarray*}
for a positive constant $C$ depending only on $R, c_v, \mu, \kappa,
m_1, N_1$, and $T$.
By Corollary \ref{COREstL2}, Corollary \ref{COREstH1}, and (ii) of
Proposition \ref{PROPEstH2-PRE}, it follows from the H\"older and Young
inequalities that
\begin{eqnarray*}
  &&\frac1\mu\int_0^L[(2JG+R\varrho_0
  \vartheta)\partial_tG
  +\partial_yvG^2]\partial_t\vartheta dy\\
  &\leq&
  C[(\|J\|_\infty\|G\|_2+\|\sqrt{\varrho_0}\vartheta\|_2)\|\partial_t
  G\|_2+\|\partial_yv\|_2\|G\|_2\|G\|_\infty]\|\partial_t
  \vartheta\|_\infty\\
  &\leq&
  C(\|\partial_tG\|_2+\|G\|_\infty)\left(\|\sqrt{\varrho_0}
  \partial_t\vartheta\|_2+\left\|\frac{\partial_y\partial_t\vartheta}
  {\sqrt J}\right\|_2\right)\\
  &\leq&\frac\kappa4\left\|\frac{\partial_y\partial_t\vartheta}{\sqrt J}
  \right\|_2^2+C\left(\|\sqrt J\partial_tG\|_2^2+
  \|\sqrt{\varrho_0}\partial_t\vartheta\|_2^2+\|G\|_\infty^2\right),
\end{eqnarray*}
for a positive constant $C$ depending only on
$R, c_v, \mu, \kappa, m_1, N_1, N_2$, and $T$. Therefore,
one obtains from (\ref{EINQGH1}) that
\begin{eqnarray}
  &&c_v\frac{d}{dt}\|\sqrt{\varrho_0}\partial_t\vartheta\|_2^2
  +\kappa\left\|\frac{\partial_y\partial_t\vartheta}{\sqrt J}
  \right\|_2^2\nonumber\\
  &\leq&A_3\left[\|\sqrt J\partial_tG\|_2^2+(1+\|G\|_\infty^2
  +\|\vartheta\|_\infty^2)
  (\|\sqrt{\varrho_0}\partial_t\vartheta\|_2^2+1)\right], \label{EINQGH1-B}
\end{eqnarray}
for a positive constant $A_3$ depending only on
$R, c_v, \mu, \kappa, m_1, N_1, N_2$, and $T$.

Using (\ref{EQtheta'}), one can rewrite (\ref{EQG}) as
\begin{equation*}
  \partial_tG-\frac\mu J\partial_y\left(\frac{\partial_yG}{\varrho_0}\right)
  =(c_v-R)\frac{\varrho_0}{J}\partial_t\vartheta
  -\frac{\partial_yv}{J}G.
\end{equation*}
Multiplying the above equation by $J\partial_tG$, integrating the
resultant over $(0,L)$, and integrating by parts, it follows from
the H\"older and Young inequalities, Corollary \ref{COREstL2},
and Corollary \ref{COREstH1} that
\begin{eqnarray*}
  &&\frac\mu2\frac{d}{dt}\int_0^L\left|\frac{\partial_yG}
  {\sqrt{\varrho_0}}
  \right|^2dy +\int_0^LJ|\partial_tG|^2dy\\
  &=&(c_v-R)\int_0^L\varrho_0\partial_t\vartheta\partial_tGdy+\int_0^L
  \partial_yvG\partial_tGdy  \\
  &\leq&\frac12\|\sqrt J\partial_tG\|_2^2+C(\|\sqrt{\varrho_0}
  \partial_t\vartheta\|_2^2+\|\partial_yv\|_2^2\|G\|_\infty^2) \\
  &\leq&\frac12\|\sqrt J\partial_tG\|_2^2+C(\|\sqrt{\varrho_0}
  \partial_t\vartheta\|_2^2+\|G\|_\infty^2)
\end{eqnarray*}
and, thus,
\begin{equation}
\label{EINQGH1-C}
  \mu\frac{d}{dt}\left\|
  \frac{\partial_yG}{\sqrt{\varrho_0}}\right\|_2^2
  +\|\sqrt J\partial_tG\|_2^2\leq C (\|\sqrt{\varrho_0}
  \partial_t\vartheta\|_2^2+\|G\|_\infty^2),
\end{equation}
for a positive constant $C$ depending only on
$R, c_v, \mu, \kappa, m_1, N_1, N_2$, and $T$.

Multiplying (\ref{EINQGH1-C}) by $2A_3$ and summing the resultant
with (\ref{EINQGH1-B}), one obtains
\begin{align*}
\frac{d}{dt}\Bigg(c_v\|\sqrt{\varrho_0}\partial_t\vartheta\|_2^2
  &+2A_3\mu\left\|
  \frac{\partial_yG}{\sqrt{\varrho_0}}\right\|_2^2\Bigg)
  +\kappa\left\|\frac{\partial_y\partial_t\vartheta}{\sqrt J}
  \right\|_2^2+A_3\|\sqrt J\partial_tG\|_2^2\nonumber\\
  \leq&\ \ C(1+\|G\|_\infty^2
  +\|\vartheta\|_\infty^2)
  (\|\sqrt{\varrho_0}\partial_t\vartheta\|_2^2+1),
\end{align*}
for a positive constant $C$ depending only on
$R, c_v, \mu, \kappa, m_1, N_1, N_2$, and $T$.
Applying the Gronwall inequality to the above inequality, by
Corollary \ref{COREstL2}, and using (\ref{EstGInf}), the
conclusion follows.
\end{proof}

\begin{proposition}
\label{PROPEstH2-B}
Given $T\in(0,\infty)$. It holds that
\begin{equation*}
  \sup_{0\leq t\leq T}\left(\left\|(\partial_y^2J,
  \partial_y^2v,\partial_y\vartheta,
  \partial_y^2\vartheta)\right\|_2^2+\|\vartheta\|_\infty\right)
  +\int_0^T\|(\partial_y^3v,\partial_y\partial_tv,
  \partial_y^3\vartheta)\|_2^2\leq C,
\end{equation*}
for a positive constant $C$ depending only on
$R, c_v, \mu, \kappa, m_1, N_1, N_2, N_3$, and $T$, where
$$
N_3:=\|\varrho_0''\|_2+\|g_0\|_2+\|h_0\|_2, 
$$
and $m_1, N_1$ and $N_2$ are the numbers in
  Propositions \ref{PROPEstJ}, \ref{PROPEstRhoWei}, and \ref{PROPEstGL2}, respectively.
\end{proposition}

\begin{proof}
Combining (i)
of Proposition \ref{PROPEstH2-PRE} and Proposition \ref{PROPEstH2-A},
one gets
\begin{equation}
\sup_{0\leq t\leq T}\|\partial_y\vartheta\|_2^2\leq C\sup_{0\leq t\leq
T}(1+\|\sqrt{\varrho_0}\partial_t\vartheta\|_2^2)
\leq C(1+\|g_0\|_2^2+\|h_0\|_2^2)\label{ESTH2-1}
\end{equation}
and, thus, by (i) of
Proposition \ref{PROPEstRhoWei} and Corollary
\ref{COREstL2} that
that
\begin{equation}
  \sup_{0\leq t\leq T}\|\vartheta\|_\infty
  \leq C\sup_{0\leq t\leq T}(\|\partial_y\vartheta\|_2+1)
  \leq C(1+\|g_0\|_2+\|h_0\|_2),\label{ESTH2-2}
\end{equation}
for a positive constant $C$ depending only on
$R, c_v, \mu, \kappa, m_1, N_1, N_2$, and $T$. Using (\ref{ESTH2-1})--(\ref{ESTH2-2}), it follows from the H\"older inequality and
Corollaries \ref{COREstL2}--\ref{COREstH1} that
\begin{eqnarray}
  \|\partial_y\pi\|_2&=&R\left\|\frac{\varrho_0'}{J}\vartheta+\frac{
  \varrho_0}{J}\partial_y\vartheta-\frac{\varrho_0}{J^2}\partial_y J
  \vartheta\right\|_2\nonumber\\
  &\leq&C(\|\varrho_0'\|_2\|\vartheta\|_\infty+\|\varrho_0\|_\infty \|\partial_y\vartheta\|_2+\|\varrho_0\|_\infty\|\partial_yJ\|_2
  \|\vartheta\|_\infty) \nonumber\\
  &\leq&C(1+\|g_0\|_2+\|h_0\|_2),\label{ESTH2-7}
\end{eqnarray}
for a positive constant $C$ depending only on
$R, c_v, \mu, \kappa, m_1, N_1, N_2$, and $T$.

Noticing that
$\partial_yv=\frac1\mu(JG+R\varrho_0\vartheta)$, and using (\ref{EQJ}),
one has
\begin{eqnarray*}
\partial_t\partial_yv=\frac1\mu(\partial_yvG+J\partial_tG+R
\varrho_0\partial_t\vartheta),\\
\partial_y^2v=\frac1\mu(\partial_yJG+J\partial_yG+R\varrho_0'\vartheta
+R\varrho_0\partial_y\vartheta),
\end{eqnarray*}
and, thus, by the H\"older and Sobolev
embedding inequalities, and using (\ref{ESTH2-1})--(\ref{ESTH2-2}),
it follows from Corollaries \ref{COREstL2}--\ref{COREstH1} and
Proposition \ref{PROPEstH2-A} that
\begin{eqnarray}
  \int_0^T\|\partial_t\partial_yv\|_2^2dt&\leq&
  C\int_0^T(\|\partial_yv\|_2^2\|G\|_\infty^2+
  \|J\|_\infty^2\|\partial_tG\|_2^2+\|\sqrt{\varrho_0}\partial_t
  \vartheta\|_2^2)dt \nonumber\\
  &\leq&C\int_0^T(\|\partial_yv\|_2^2\|G\|_{H^1}^2+
  \|J\|_\infty^2\|\partial_tG\|_2^2+\|\sqrt{\varrho_0}\partial_t
  \vartheta\|_2^2)dt\nonumber\\
  &\leq&C(1+\|g_0\|_2^2+\|h_0\|_2^2), \label{ESTH2-3}
\end{eqnarray}
and
\begin{eqnarray}
  \sup_{0\leq t\leq T}\|\partial_y^2v\|_2^2
  &\leq&C\sup_{0\leq t\leq T}(\|\partial_yJ\|_2^2\|G\|_\infty^2+
  \|J\|_\infty^2\|\partial_yG\|_2^2+\|\varrho_0'\|_2^2
  \|\vartheta\|_\infty^2+\|\partial_y\vartheta\|_2^2)\nonumber\\
  &\leq&C\sup_{0\leq t\leq T}(\|\partial_yJ\|_2^2\|G\|_{H^1}^2+
  \|J\|_\infty^2\|\partial_yG\|_2^2+
  \|\vartheta\|_\infty^2+\|\partial_y\vartheta\|_2^2)\nonumber\\
  &\leq&C(1+\|g_0\|_2^2+\|h_0\|_2^2),\label{ESTH2-4}
\end{eqnarray}
for a positive constant $C$ depending only on
$R, c_v, \mu, \kappa, m_1, N_1, N_2$, and $T$.

Using (\ref{EQtheta'}),
we have
\begin{equation*}
  \partial_y^2\vartheta=J\partial_y\left(\frac{\partial_y\vartheta}{J}
  \right)+\partial_yJ\frac{\partial_y\vartheta}{J}
  =\frac J\kappa(c_v\varrho_0\partial_t\vartheta-\partial_yv G)+
  \partial_yJ\frac{\partial_y\vartheta}{J}
\end{equation*}
and, thus, by the H\"older, Young
and Gagliardo-Nirenburg inequalities and (\ref{ESTH2-1}),
it follows from Corollaries \ref{COREstL2}--\ref{COREstH1}, and Proposition \ref{PROPEstH2-A} that
\begin{eqnarray}
  \|\partial_y^2\vartheta\|_2&\leq& C(\|\sqrt{\varrho_0}\partial_t\vartheta\|_2+\|\partial_yv\|_2
  \|G\|_\infty+\|\partial_yJ\|_2
  \|\partial_y\vartheta\|_\infty)\nonumber\\
  &\leq& C(1+\|g_0\|_2+\|h_0\|_2+\|G\|_{H^1}+\|\partial_y\vartheta\|_2^{\frac12}
  \|\partial_y^2\vartheta\|_2^{\frac12})\nonumber\\
  &\leq&\frac12\|\partial_y^2\vartheta\|_2+C(1+\|g_0\|_2+\|h_0\|_2),
  \nonumber
\end{eqnarray}
which gives
\begin{equation}
  \sup_{0\leq t\leq T}\|\partial_y^2\vartheta\|_2^2\leq C(1+\|g_0\|_2^2+\|h_0\|_2^2),  \label{ESTH2-5}
\end{equation}
for a positive constant $C$ depending only on
$R, c_v, \mu, \kappa, m_1, N_1, N_2$, and $T$.

By calculations, one deduces
\begin{eqnarray*}
  \partial_y^2\pi&=&R\partial_y^2\left(\frac{\varrho_0}{J}
  \vartheta\right)
  =R\left[\varrho_0''\frac{\vartheta}{J}+2\varrho_0'\partial_y
  \left(\frac{\vartheta}{J}\right)+\varrho_0\partial_y^2
  \left(\frac\vartheta J\right)\right]\\
  &=&R\left[\varrho_0''\frac{\vartheta}{J}
  +2\varrho_0'\left(\frac{\partial_y \vartheta}{J}-\frac{\partial_yJ}{J^2}\vartheta\right)\right.\\
  &&\left.+\varrho_0\left(\frac{\partial_y^2\vartheta}{J}-\frac{2}{J^2} \partial_y J\partial_y\vartheta+2\frac{(\partial_yJ)^2}{J^3}\vartheta
  -\frac{\partial_y^2J}{J^2}\vartheta\right)\right].
\end{eqnarray*}
Therefore, by the H\"older and Sobolev embedding inequalities,
using (\ref{ESTH2-1}), (\ref{ESTH2-2}), and (\ref{ESTH2-5}), it follows
from Corollary \ref{COREstL2} and Corollary \ref{COREstH1} that
\begin{eqnarray}
  \|\partial_y^2\pi\|_2&\leq&C\Big[\|\varrho_0''\|_2
  \left\| \vartheta \right\|_\infty+2\|\varrho_0'\|_\infty
  \left(\left\| \partial_y\vartheta \right\|_2
  +\left\| \partial_y J \right\|_2\|\vartheta\|_\infty\right)
   \nonumber\\
  &&+\|\varrho_0\|_\infty\left(\left\| \partial_y^2\vartheta
  \right\|_2+2\left\| \partial_y J \right\|_\infty \|\partial_y\vartheta\|_2+2\left\| \partial_yJ \right\|_\infty
  \|\partial_yJ\|_2\|\vartheta\|_\infty\right)\nonumber\\
  && +\|\varrho_0\|_\infty\|\partial_y^2J\|_2\left\|
   \vartheta
  \right\|_\infty\Big] \leq C(1+\|g_0\|_2+\|h_0\|_2
  +\|\partial_y^2J
  \|_2),\label{ESTH2-6}
\end{eqnarray}
for a positive constant $C$ depending only on
$R, c_v, \mu, \kappa, m_1, N_1, N_2, N_3$, and $T$.

Using (\ref{EQtheta'}) and (\ref{EQv}), we deduce
\begin{eqnarray*}
  \partial_y^3\vartheta&=&
  \partial_y^2\left(\frac{\partial_y\vartheta}{J}\right) J
  +2\partial_y\left(\frac{\partial_y\vartheta}{J}\right)\partial_yJ
  +\frac{\partial_y\vartheta}{J}\partial_y^2J\\
  &=&\frac J\kappa\left[c_v\left(\varrho_0\partial_t\partial_y\vartheta
  +\varrho_0'\partial_t\vartheta\right)
  -\partial_yv\partial_yG-\partial_y^2vG\right]\\
  &&+\frac2\kappa\partial_yJ(c_v\varrho_0\partial_t\vartheta-\partial_y
  vG)+\frac{\partial_y\vartheta}{J}\partial_y^2J
\end{eqnarray*}
and
\begin{eqnarray*}
  \partial_y^3v&=&
  \partial_y^2\left(\frac{\partial_yv}{J}\right) J
  +2\partial_y\left(\frac{\partial_yv}{J}\right)\partial_yJ
  +\frac{\partial_yv}{J}\partial_y^2J\\
  &=&\frac J\mu(\varrho_0'\partial_tv+\varrho_0\partial_t\partial_yv
  +\partial_y^2\pi)+\frac2\mu\partial_yJ(\varrho_0\partial_tv
  +\partial_y\pi)+\frac{\partial_yv}{J}\partial_y^2J.
\end{eqnarray*}
Therefore, by the H\"older and Sobolev embedding inequalities, using
(\ref{ESTH2-1}), (\ref{ESTH2-7}), (\ref{ESTH2-4}), (\ref{ESTH2-5}),
(\ref{ESTH2-6}),
Corollary
\ref{COREstL2}, Corollary \ref{COREstH1}, (ii) of
Proposition \ref{PROPEstH2-PRE},
and Proposition \ref{PROPEstH2-A}, we deduce
\begin{eqnarray}
  \|\partial_y^3\vartheta\|_2&\leq&C[\|\partial_y\partial_t
  \vartheta\|_2
  +\|\varrho_0'\|_2\|\partial_t\vartheta\|_\infty+\|\partial_yv
  \|_\infty\|\partial_y
  G\|_2+\|\partial_y^2v\|_2\|G\|_\infty\nonumber\\
  &&+\|\partial_yJ\|_\infty(\|\sqrt{
  \varrho_0}\partial_t\vartheta\|_2+\|\partial_yv\|_2\|G\|_\infty)
  +\|\partial_y\vartheta\|_\infty\|\partial_y^2J\|_2]\nonumber\\
  &\leq&C[\|\partial_y\partial_t
  \vartheta\|_2
  +\|\sqrt{\varrho_0}\partial_t\vartheta\|_2+
  \|\partial_yv\|_{H^1}\|\partial_y
  G\|_2+\|\partial_y^2v\|_2\|G\|_{H^1}\nonumber\\
  &&+\|\partial_yJ\|_{H^1}(\|\sqrt{
  \varrho_0}\partial_t\vartheta\|_2+\|\partial_yv\|_2\|G\|_{H^1})
  +\|\partial_y\vartheta\|_{H^1}\|\partial_y^2J\|_2]\nonumber\\
  &\leq&C(1+\|g_0\|_2+\|h_0\|_2+\|\partial_y\partial_t\vartheta\|_2+
  \|\partial_y^2J\|_2)\nonumber\\
  &\leq&C(1+\|g_0\|_2+\|h_0\|_2)(1+\|\partial_y\partial_t\vartheta\|_2+
  \|\partial_y^2J\|_2), \label{ESTH2-8}
\end{eqnarray}
and
\begin{eqnarray}
  \|\partial_y^3v\|_2&\leq&
  C[\|\partial_tv\|_\infty+\|\partial_y\partial_tv\|_2
  +\|\partial_y^2\pi\|_2
  \nonumber\\
  &&+\|\partial_yJ\|_\infty(\|\sqrt{\varrho_0}\partial_tv\|_2+\|\partial_y
  \pi\|_2)+\|\partial_yv\|_\infty\|\partial_y^2 J\|_2]\nonumber\\
  &\leq&C[\|\partial_y\partial_tv\|_2+\|\sqrt{\varrho_0}\partial_tv\|_2+
  1+\|g_0\|_2+\|h_0\|_2+\|\partial_y^2J\|_2\nonumber\\
  &&+\|\partial_yJ\|_{H^1}(1+\|g_0\|_2+\|h_0\|_2)+\|\partial_yv
  \|_{H^1}\|\partial_y^2 J\|_2]\nonumber\\
  &\leq&C(1+\|g_0\|_2+\|h_0\|_2)(1+\|\partial_y\partial_tv\|_2+
  \|\partial_y^2J\|_2),\label{ESTH2-9}
\end{eqnarray}
for a positive constant $C$ depending only on
$R, c_v, \mu, \kappa, m_1, N_1, N_2, N_3$, and $T$.

Combining (\ref{ESTH2-8}) with (\ref{ESTH2-9}), and using (\ref{ESTH2-3}), one obtains
\begin{eqnarray}
  \int_0^t\|(\partial_y^3v,\partial_y^3\vartheta)\|_2^2d\tau
  &\leq& C(1+\|g_0\|_2^2+\|h_0\|_2^2)\int_0^t(1+\|\partial_y\partial_t
  v\|_2^2+\|\partial_y^2J\|_2^2)d\tau \nonumber\\
  &\leq& C(1+\|g_0\|_2^2+\|h_0\|_2^2)^2\left(1+
  \int_0^t\|\partial_y^2J\|_2^2d\tau\right), \label{ESTH2-10}
\end{eqnarray}
for any $t\in[0,T]$, where $C$ is a positive constant depending only on
$R, c_v, \mu, \kappa, m_1$, $N_1$, $N_2$, $N_3$, and $T$.
Using (\ref{EQJ}), one gets $J=1+\int_0^t\partial_yvd\tau$ and, thus,
it follows from the H\"older inequality that
\begin{eqnarray*}
  \|\partial_y^2J\|_2^2(t)&=&\left(\left\|
  \int_0^t\partial_y^3vd\tau\right\|_2\right)^2\leq \left(\int_0^t\|\partial_y^3v\|_2d\tau\right)^2 \leq t
  \int_0^t\|\partial_y^3v\|_2^2d\tau.
\end{eqnarray*}
Combining this with (\ref{ESTH2-10}), and applying the Gronwall inequality, one obtains
\begin{eqnarray*}
\int_0^T(\|\partial_y^3v\|_2^2+\|\partial_y^3\vartheta\|_2^2)dt\leq C
\end{eqnarray*}
and, further, that
\begin{eqnarray*}
  \sup_{0\leq t\leq C}\|\partial_y^2J\|_2^2\leq C,
\end{eqnarray*}
for a positive constant $C$ depending only on
$R, c_v, \mu, \kappa, m_1$, $N_1$, $N_2$, $N_3$, and $T$.
\end{proof}

We summarize the a priori estimates obtained in this section as:

\begin{corollary}
  \label{CORESTALL}
  Given $T\in(0,T)$ and let $m_1, N_1, N_2$ and $N_3$ be the numbers in
  Propositions \ref{PROPEstJ}, \ref{PROPEstRhoWei}, \ref{PROPEstGL2}, and \ref{PROPEstH2-B}, respectively. Then, there are two positive constants $\underline C$ and $C$ depending only on $R, c_v, \mu, \kappa, m_1$, $N_1$, $N_2$, $N_3$, and $T$, such that
  $$
  \inf_{0\leq t\leq T}\inf_{y\in(0,L)}J(y,t)\geq \underline C,
  $$
  and
  \begin{eqnarray*}
    \sup_{0\leq t\leq T}\|(J,v,\vartheta)\|_{H^2}^2+\int_0^T
    (\|\partial_tJ\|_{H^2}^2+\|(v,\vartheta)\|_{H^3}^2
    +\|(\partial_tv,\partial_t\vartheta)\|_{H^1}^2)dt\leq C.
  \end{eqnarray*}
\end{corollary}

\begin{proof}
All the estimates except
$$
\sup_{0\leq t\leq T}\|v\|_2^2+\int_0^T(\|\partial_tv\|_2^2+\|\partial_t
\vartheta\|_2^2+\|\partial_tJ\|_{H^2}^2)dt
$$
are directly corollaries of Corollary \ref{COREstL2}, Corollary
\ref{COREstH1}, Proposition \ref{PROPEstH2-A}, and Proposition
\ref{PROPEstH2-B}. While the remaining estimates in the above follow
easily from the known ones by the Poincar\'e inequality, using equation
(\ref{EQJ}), or (ii) of Proposition \ref{PROPEstH2-PRE}.
\end{proof}

\section{Proof of Theorem \ref{thm}}

\begin{proof}[\textbf{Proof of Theorem \ref{thm}}]
For $\varepsilon\in(0,1)$, set
$$
\varrho_{0\varepsilon}=\varrho_0+\varepsilon,\quad
\vartheta_{0\varepsilon}=\vartheta_0+\varepsilon.
$$
Let $E_{0\varepsilon}, m_{1\varepsilon}$, and $N_{i\varepsilon}$,
$i=1,2,3,$ be the corresponding numbers as stated in Section
\ref{SecApri} for $(\varrho_{0\varepsilon}, v_0,
\vartheta_{0\varepsilon})$, that is 
\begin{eqnarray*}
E_{0\varepsilon}:=
\int_0^L\left(\frac{\varrho_{0\varepsilon}}{2}v_0^2
+c_v\varrho_{0\varepsilon}\vartheta_{0\varepsilon}\right)dy,\quad m_{1\varepsilon}=\exp\left\{\frac2\mu\sqrt{2\|\varrho_{0\varepsilon}
\|_1E_{0\varepsilon}}\right\},
\end{eqnarray*}
and $N_{i\varepsilon}$,
$i=1,2,3,$ will be given later. 
We first verify that all these numbers
are uniformly bounded.
One can easily check that
\begin{eqnarray*}
  \|\varrho_0\|_\infty\leq\|\varrho_{0\varepsilon}\|_\infty\leq \|\varrho_0\|_\infty+1, \quad \|\varrho_0\|_1\leq\|\varrho_{0\varepsilon}\|_1\leq \|\varrho_0\|_1+L, \\
  E_0\leq E_{0\varepsilon}\leq E_0+\|v_0\|_2^2+c_v(\|\varrho_0\|_1+
  \|\vartheta_0\|_1+L),
\end{eqnarray*}
for $\varepsilon\in(0,1)$.
Noticing that
\begin{eqnarray*}
  \Omega_0:=\left\{y\in(0,L)\Bigg|\varrho_0(y)\geq
  \frac{\bar\varrho}{2}
  \right\}\subseteq\Omega_{0\varepsilon}:= \left\{y\in(0,L)\Bigg|\varrho_{0\varepsilon}(y)\geq
  \frac{\bar\varrho}{2}
  \right\},
\end{eqnarray*}
we have
$$
0<\omega_0:=|\Omega_0|\leq
\omega_{0\varepsilon}:=|\Omega_{0\varepsilon}|\leq L,
$$
for $\varepsilon\in(0,1)$.
Therefore, we have
$$
\underline m_1\leq m_{1\varepsilon}\leq\bar m_1,
$$
and
\begin{eqnarray*}
  N_{1\varepsilon}:=\frac{E_{0\varepsilon}}{c_v} +\|\varrho_{0\varepsilon}\|_\infty
  +\frac{1}{\|\varrho_{0\varepsilon}\|_\infty}
  +L+\frac1L+\frac{1}{\omega_{0\varepsilon}}+\|\varrho_0'\|_\infty\leq \bar N_1,
\end{eqnarray*}
for some positive constants $\underline m_1, \bar m_1,$ and $\bar N_1$ independent of $\varepsilon\in(0,1)$.
We have
\begin{eqnarray*}
  N_{2\varepsilon}&:=&\|\sqrt{\varrho_{0\varepsilon}}v_0^2\|_2
  +\|\sqrt{\varrho_{0\varepsilon}}v_0\|_2+\|v_0'\|_2 \\
  &\leq&\|\sqrt{\varrho_0}v_0^2\|_2+\sqrt\varepsilon\|v_0^2\|_2
  +\|\sqrt{\varrho_0}v_0\|_2+\sqrt\varepsilon\|v_0\|_2+\|v_0'\|_2\\
  &\leq&\|\sqrt{\varrho_0}v_0^2\|_2+\|v_0^2\|_2+
  \|\sqrt{\varrho_0}v_0\|_2+\|v_0\|_2+\|v_0'\|_2,
\end{eqnarray*}
for any $\varepsilon\in(0,1)$. Set
$$
N_{3\varepsilon}:=\|g_{0\varepsilon}\|_2+\|h_{0\varepsilon}\|_2
+\|\varrho_{0\varepsilon}''\|_2,
$$
where
$$
g_{0\varepsilon}=\frac{\mu v_0''-R(\varrho_{0\varepsilon}\vartheta_{0\varepsilon})'}
{\sqrt{\varrho_{0\varepsilon}}}
,\quad h_{0\varepsilon}=\frac{1}{\sqrt{\varrho_{0\varepsilon}}}
[\mu(v_0')^2+\kappa\vartheta_{0\varepsilon}''-
Rv_0'\varrho_{0\varepsilon}\vartheta_{0\varepsilon}].
$$
By direct calculations, and using the compatibility conditions, we have
\begin{eqnarray*}
  \|g_{0\varepsilon}\|_2&=&\left\|\frac{1}{\sqrt{\varrho_{0\varepsilon}}}
  (\sqrt{\varrho_0} g_0+\varepsilon(\varrho_0'+\vartheta_0'))\right\|_2\\
  &\leq&\|g_0\|_2+\sqrt\varepsilon(\|\varrho_0'\|_2+\|\vartheta_0'\|_2)
  \leq\|g_0\|_2+ \|\varrho_0'\|_2+\|\vartheta_0'\|_2,
\end{eqnarray*}
and
\begin{eqnarray*}
  \|h_{0\varepsilon}\|_2&=&\left\|\frac{1}{\sqrt{\varrho_{0\varepsilon}}}
  [\sqrt{\varrho_0}h_0-Rv_0'(\varepsilon\varrho_0+\varepsilon\vartheta_0
  +\varepsilon^2)\right\|_2\\
  &\leq&\|h_0\|_2+R\|v_0'\|_2(\sqrt\varepsilon\|\varrho_0\|_\infty+
  \sqrt\varepsilon\|\vartheta_0\|_\infty +\varepsilon^{\frac32})\\
  &\leq&\|h_0\|_2+R\|v_0'\|_2(\|\varrho_0\|_\infty+
  \|\vartheta_0\|_\infty +1),
\end{eqnarray*}
for any $\varepsilon\in(0,1)$. Therefore, we have
$$
N_{3\varepsilon}\leq \bar N_3,
$$
for some positive constant $\bar N_3$ independent of $\varepsilon\in(0,1)$.

By Proposition \ref{PROPGLOBAL}, for each $\varepsilon\in(0,1)$, there
is a unique global strong solution $(J_\varepsilon, v_\varepsilon,
\vartheta_\varepsilon)$ to system (\ref{EQJ})--(\ref{EQtheta}),
with $\varrho_0$ replaced by $\varrho_{0\varepsilon}$,
subject to (\ref{BC}) and the initial condition
$$
(J_\varepsilon, v_\varepsilon, \theta_\varepsilon)|_{t=0}=(1, v_0,
\theta_{0\varepsilon}).
$$
Due to the uniform boundedness of $m_{1\varepsilon}$, $N_{1\varepsilon}, N_{2\varepsilon},$ and $N_{3\varepsilon}$, obtained in the above, it
follows from Corollary \ref{CORESTALL} that there are
two positive constants, independent of $\varepsilon\in(0,1)$, such
that
\begin{equation}
  \inf_{0\leq t\leq T}\inf_{y\in(0,L)}J_\varepsilon
  (y,t)\geq \underline C,\label{4.1}
\end{equation}
and
\begin{eqnarray}
    \sup_{0\leq t\leq T}\|(J_\varepsilon,v_\varepsilon,\vartheta_\varepsilon)
    \|_{H^2}^2+\int_0^T
    (\|\partial_tJ_\varepsilon\|_{H^2}^2+\|(v_\varepsilon,
    \vartheta_\varepsilon)\|_{H^3}^2
    +\|(\partial_tv_\varepsilon,\partial_t\vartheta_\varepsilon)
    \|_{H^1}^2)dt\leq C,\label{4.2}
\end{eqnarray}
for any $\varepsilon\in(0,1)$.

Thanks to (\ref{4.2}), by the Banach-Alaoglu theorem, and using Cantor's
diagonal arguments, there is a subsequence, still denoted by $\{(J_\varepsilon, v_\varepsilon, \theta_\varepsilon)\}$, and
a triple $(J,v,\vartheta)$, such that
\begin{eqnarray}
  (J_\varepsilon, v_\varepsilon, \vartheta_\varepsilon)\overset{*}{\rightharpoonup}(J,v,\vartheta), &&
  \mbox{in }L^\infty(0,T; H^2),\label{CG1}\\
  \partial_tJ_\varepsilon\overset{*}{\rightharpoonup}\partial_tJ,
  &&\mbox{in }L^2(0,T; H^2),\label{CG2}\\
  (\partial_tv_\varepsilon,\partial_t\vartheta_\varepsilon)\rightharpoonup
  (\partial_tv,\partial_t\vartheta),&&\mbox{in }L^2(0,T;H^1),\label{CG3}\\
  (v_\varepsilon,\vartheta_\varepsilon)\rightharpoonup(v,\vartheta),&&
  \mbox{in }L^2(0,T; H^3),\label{CG4}
\end{eqnarray}
and, moreover, by Aubin-Lions compactness lemma, that
\begin{eqnarray}
  (J_\varepsilon,v_\varepsilon,\vartheta_\varepsilon)\rightarrow
  (J,v,\vartheta),&&\mbox{in }C([0,T];H^1),\label{CG5}\\
  (v_\varepsilon,\vartheta_\varepsilon)\rightarrow
  (v,\vartheta),&&\mbox{in }L^2(0,T;H^2),\label{CG6}
\end{eqnarray}
where $\rightharpoonup, \overset{*}{\rightharpoonup}$, and $\rightarrow$, respectively, denote the weak, weak$-*$, and strong
convergences in the corresponding spaces.
Noticing that $H^1((0,L))\hookrightarrow C([0,L])$, we then have
$$
J_\varepsilon\rightarrow J,\quad\mbox{in }C([0,L]\times[0,T])
$$
and, thus, it follow follows from (\ref{4.1}) that
\begin{equation*}
  \inf_{0\leq t\leq T}\inf_{y\in(0,L)}J(y,t)>0,
\end{equation*}
for any $T\in(0,\infty)$.

Thanks to the convergences (\ref{CG1})--(\ref{CG6}), one can take the
limit $\varepsilon\rightarrow0$ to show that $(J,v,\vartheta)$ is a
strong solution to system (\ref{EQJ})--(\ref{EQtheta}), subject to
(\ref{BC})--(\ref{IC}), satisfying the regularities stated in Theorem
\ref{thm}.

We now prove the uniqueness. Let $(J_1, v_1, \vartheta_1)$ and $(J_2,
v_2, \vartheta_2)$ be two solutions to system
(\ref{EQJ})--(\ref{EQtheta}), subject to (\ref{BC})--(\ref{IC}),
satisfying the regularities stated in Theorem \ref{thm}, with the
same initial data. Denote by $(J, v,\vartheta)$ the difference of
these two solutions, that is,
$$
(J, v,\vartheta)=(J_1, v_1,\vartheta_1)-(J_2, v_2,\vartheta_2).
$$
Then, $(J, v, \vartheta)$ satisfies the following
\begin{eqnarray}
  \partial_tJ&=&\partial_yv,\label{D1}\\
  \varrho_0\partial_tv-\mu\partial_y\left(\frac{\partial_yv}{J_1}\right)
  &=&-\mu\partial_y\left(\frac{\partial_yv_2}{J_1J_2}J\right)+R\partial_y
  \left(\frac{\varrho_0\vartheta_2}{J_1J_2}J-\frac{\varrho_0}{J_1}
  \vartheta
  \right),\label{D2}\\
  c_v\varrho_0\partial_t\vartheta-\kappa\partial_y\left(\frac{\partial_y
  \vartheta}{J_1}\right)&=&-\kappa\partial_y\left(\frac{\partial_y
  \vartheta_2}
  {J_1J_2}J\right)+\left[\mu\left(\frac{\partial_yv_1}{J_1}
  +\frac{\partial_y
  v_2}{J_2}\right)-R\frac{\varrho_0\vartheta_2}{J_2}\right]\partial_yv
  \nonumber\\
  &&-R\frac{\varrho_0\partial_yv_1}{J_1}\vartheta-\frac{\partial_yv_1}
  {J_1J_2}
  (\mu\partial_yv_2-R\varrho_0\vartheta_2)J.\label{D3}
\end{eqnarray}

Multiplying (\ref{D2}) by $v$, integrating the resultant over $(0,L)$,
and integrating by parts, it follows from the Young and
Sobolev embedding inequalities, and the regularities of $(J_i, v_i,
\vartheta_i)$, $i=1,2,$ that
\begin{eqnarray*}
   \frac12\frac{d}{dt}\|\sqrt{\varrho_0}v\|_2^2+\mu\left\|\frac{\partial_y
  v}{\sqrt{J_1}}\right\|_2^2
  &=&\int_0^L\left[
  \mu\frac{\partial_yv_2}{J_1J_2}J+R
  \left(\frac{\varrho_0}{J_1}\vartheta-
  \frac{\varrho_0\vartheta_2}{J_1J_2}J
  \right)
  \right]\partial_yvdy\\
  &\leq&\frac\mu2\left\|\frac{\partial_y
  v}{\sqrt{J_1}}\right\|_2^2+C\int_0^L\left(
  \frac{|\partial_yv_2|^2
  +\varrho_0^2\vartheta_2^2}{J_1J_2^2}J^2
  +\frac{\varrho_0^2\vartheta^2}{J_1}\right)dy\\
  &\leq&\frac\mu2\left\|\frac{\partial_y
  v}{\sqrt{J_1}}\right\|_2^2+C(\|(\partial_yv_2,\vartheta_2)
  \|_\infty^2+1)\|(J,\sqrt{\varrho_0}\vartheta)\|_2^2\\
  &\leq&\frac\mu2\left\|\frac{\partial_y
  v}{\sqrt{J_1}}\right\|_2^2+C (\|J\|_2^2+\|\sqrt{\varrho_0}\vartheta\|_2^2)
\end{eqnarray*}
and, thus,
\begin{equation}
  \label{ESTD1}
  \frac{d}{dt}\|\sqrt{\varrho_0}v\|_2^2+\mu\left\|\frac{\partial_y
  v}{\sqrt{J_1}}\right\|_2^2
  \leq C (\|J\|_2^2+\|\sqrt{\varrho_0}\vartheta\|_2^2).
\end{equation}

Multiplying (\ref{D3}) by $\vartheta$, integrating the resultant over
$(0,L)$, and integrating by parts, it follows from the H\"oler, Young,
and Sobolev embedding inequalities, and the regularities of $(J_i, v_i,
\vartheta_i)$, $i=1,2,$ that
\begin{eqnarray*}
  &&\frac{c_v}{2}\frac{d}{dt}\|\sqrt{\varrho_0}\vartheta\|_2^2
  +\kappa\left\|\frac{\partial_y\vartheta}{\sqrt{J_1}}\right\|_2^2 \\
  &=&\int_0^L\left\{\kappa\frac{\partial_y\vartheta_2}{J_1J_2}J \partial_y\vartheta
  +\left[\mu\left(\frac{\partial_yv_1}{J_1}+\frac{\partial_y
  v_2}{J_2}\right)-R\frac{\varrho_0\vartheta_2}{J_2}\right]\partial_yv
  \vartheta\right\} dy \\
  &&-R\int_0^L\frac{\varrho_0\partial_yv_1}{J_1}\vartheta^2dy
  -\int_0^L\frac{\partial_yv_1}{J_1J_2}
  (\mu\partial_yv_2-R\varrho_0\vartheta_2)J\vartheta dy\\
  &\leq& C\left\|\frac{\partial_y\vartheta}{\sqrt{J_1}}\right\|_2
  \|\partial_y\vartheta_2\|_\infty\|J\|_2+C(\|\partial_yv_1\|_2
  +\|\partial_yv_2\|_2+\|\vartheta_2\|_2)\left\|\frac{\partial_yv}
  {\sqrt{J_1}}\right\|_2\|\vartheta\|_\infty \\
  &&+C\|\partial_yv_1\|_\infty\|\sqrt{\varrho_0}\vartheta\|_2^2+
  C\|\partial_yv_1\|_\infty(\|\partial_yv_2\|_2+\|\vartheta_2\|_2)
  \|J\|_2\|\vartheta\|_\infty \\
  &\leq&\frac\kappa4\left\|\frac{\partial_y\vartheta}
  {\sqrt{J_1}}\right\|_2^2+C(\|J\|_2^2+\|\sqrt{\varrho_0}\vartheta\|_2^2)
  +\left(\|J\|_2+\left\|\frac{\partial_yv}
  {\sqrt{J_1}}\right\|_2\right)\|\vartheta\|_\infty,
\end{eqnarray*}
from which, one obtains
\begin{equation}
  {c_v} \frac{d}{dt}\|\sqrt{\varrho_0}\vartheta\|_2^2
  +\frac{3\kappa}{4}\left\|\frac{\partial_y\vartheta}{\sqrt{J_1}}
  \right\|_2^2 \leq C\|(J,\sqrt{\varrho_0}\vartheta)\|_2^2
  +\left(\|J\|_2+\left\|\frac{\partial_yv}
  {\sqrt{J_1}}\right\|_2\right)\|\vartheta\|_\infty.\label{ESTD2-0}
\end{equation}
Using the same arguments as those for the estimate $\|\vartheta\|_\infty$ in (i) of Proposition \ref{PROPEstRhoWei}, one can show that
\begin{equation}
\label{O}
\|\vartheta\|_\infty\leq C\left(\|\sqrt{\varrho_0}\vartheta\|_2+
\left\|\frac{\partial_y\vartheta}
  {\sqrt{J_1}}\right\|_2\right).
\end{equation}
By the aid of this and using the Young inequality, one gets from (\ref{ESTD2-0}) that
\begin{equation}
  {c_v} \frac{d}{dt}\|\sqrt{\varrho_0}\vartheta\|_2^2
  + \kappa \left\|\frac{\partial_y\vartheta}{\sqrt{J_1}}\right\|_2^2 \leq A\left(\|J\|_2^2+\|\sqrt{\varrho_0}\vartheta\|_2^2
  +\left\|\frac{\partial_yv}
  {\sqrt{J_1}}\right\|_2^2\right),\label{ESTD2}
\end{equation}
where $A$ is a positive constant.

Multiplying (\ref{D1}) by $2J$, and integrating the resultant over $(0,L)$, it follows from the Young inequality that
\begin{equation}
\frac{d}{dt}\|J\|_2^2\leq A\left\|\frac{\partial_yv}
  {\sqrt{J_1}}\right\|_2^2+C\|J\|_2^2.\label{ESTD3}
\end{equation}

Multiplying (\ref{ESTD1}) by $\frac{2A}{\mu}$, summing the resultant with (\ref{ESTD2}) and (\ref{ESTD3}), one obtains
\begin{align*}
  \frac{d}{dt}\left(\frac{2A}{\mu}\|\sqrt{\varrho_0}v\|_2^2
  +c_v\|\sqrt{\varrho_0}\vartheta\|_2^2+\|J\|_2^2\right)
  &+ A\left\|\frac{\partial_yv}
  {\sqrt{J_1}}\right\|_2^2\\
  +\kappa \left\|\frac{\partial_y\vartheta}{\sqrt{J_1}}\right\|_2^2
   \leq& C\|(J\|_2^2+\|\sqrt{\varrho_0}\vartheta\|_2^2),
\end{align*}
from which, by the Gronwall inequality, one gets
$$
\sqrt{\varrho_0}v\equiv\sqrt{\varrho_0}\vartheta\equiv J
\equiv\partial_yv\equiv\partial_y\vartheta\equiv0.
$$
Therefore, recalling (\ref{O}) and its counterpart for $v$, we have
$$
J\equiv v\equiv\vartheta\equiv0.
$$
This proves the uniqueness.
\end{proof}

\section*{Acknowledgments}
This work was supported in part by
the the Hong Kong RGC grant CUHK 14302917, and the Direct Grant for Research 2016/2017 (Project Code: 4053216) from The
Chinese University of Hong Kong.

\end{document}